\RequirePackage{fix-cm}
\documentclass[smallextended]{svjour3}       
\smartqed  
\usepackage{graphicx}
\usepackage{color,varioref}
\usepackage{longtable}
\usepackage{subfiles}
\usepackage{multirow,array}
\usepackage[figuresright]{rotating}
\usepackage{amsfonts}
\usepackage{mathrsfs}
\usepackage{amsmath}
\usepackage{amssymb}
\usepackage{pdflscape}

\usepackage{algorithmic}
\usepackage{algorithm}
\usepackage[figuresright]{rotating}
\usepackage{subfigure}
\usepackage{appendix}

\usepackage{latexsym}
\def\cT{\mathcal{T}}

\def\cI{\mathcal{I}}

\def\diag{{\rm diag}}

\definecolor{red}{rgb}{0.9,0,0}

\definecolor{blue}{rgb}{0,0,0.9}

\begin{document}

\title{Learning the hub graphical Lasso model with the structured sparsity via an efficient algorithm}

\subtitle{}

\titlerunning{Learning the hub graphical Lasso model}        

\author{Chengjing Wang \and Peipei Tang \and Wenling He \and Meixia Lin
}

\authorrunning{C.J. Wang \and P.P. Tang} 

\institute{Chengjing Wang \at School of Mathematics, Southwest Jiaotong University \email{renascencewang@hotmail.com} \and
  Peipei Tang (corresponding author) \at
              School of Computer and
		Computing Science, Zhejiang University City College, \email{tangpp@hzcu.edu.cn} \and
  Wenling He \at School of Mathematics, Southwest Jiaotong University \email{1165649932@qq.com} \and
  Meixia Lin \at Engineering Systems and Design, Singapore University of Technology and Design \email{meixia\_lin@sutd.edu.sg}
}

\date{Received: date / Accepted: date}

\maketitle

\begin{abstract}
Graphical models have exhibited their performance in numerous tasks ranging from biological analysis to recommender systems. However, graphical models with hub nodes are computationally difficult to fit, particularly when the dimension of the data is large. To efficiently estimate the hub graphical models, we introduce a two-phase algorithm. The proposed algorithm first generates a good initial point via a dual alternating direction method of multipliers (ADMM), and then warm starts a semismooth Newton (SSN) based augmented Lagrangian method (ALM) to compute a solution that is accurate enough for practical tasks. We fully excavate the sparsity structure of the generalized Jacobian arising from the hubs in the graphical models, which ensures that the algorithm can obtain a nice solution very efficiently. Comprehensive experiments on both synthetic data and real data show that it obviously outperforms the existing state-of-the-art algorithms. In particular, in some high dimensional tasks, it can save more than 70\% of the execution time, meanwhile still achieves a high-quality estimation.
\end{abstract}

\section{Introduction}
\label{sec:Introduction}
In recent years, graph signal processing is an emerging field that plays a crucial role in various domains due to its ability to analyze and process data represented as graphs, enabling better insights, decision-making, and the development of advanced algorithms and applications. Graphical models have been widely adopted to characterize relationships of data variables in various applications, such as gene finding, medical diagnosis, modeling of protein structures, and so on.
To encode conditional dependence relationships among a set of variables, the Gaussian graphical model is employed to estimate the precision matrix $\Theta = \Sigma^{-1}$ of the data distribution, where $\Sigma$ is the corresponding covariance matrix, and then the sparsity pattern of $\Theta$ is used to determine the conditional independence graph~\cite{cox2014multivariate}. In particular, each variable in the Gaussian graphical model represents a node in a graph, and an edge will link two nodes if the corresponding entry of the precision matrix $\Theta$ is non-zero. A series of variants of Gaussian graphical model have been designed to satisfy the different requirements of practical tasks, one may refer to \cite{chaturvedi2020learning,Defazio2012convex,honorio2009sparse,hosseini2016learning,li2018learning,Meng2014Latent,molstad2018shrinking,ravikumar2008model,ShiWT,shi2024simultaneous,de2019learning,tan2014learning,tarzanagh2018estimation,yuan2007model}, to name but a few. For example, the low-rank model \cite{shi2024simultaneous} explores overall network structure by promoting sparsity and low-rankness, and the hub graphical model \cite{li2018learning,tan2014learning} studies identifying influential hub nodes that play a central role in connectivity.

In this work, we focus on the interesting problem to construct a graph with \emph{hubs}, i.e., learning a graphical model that contains a few special nodes, which are \emph{densely-connected} to several other nodes. Such kind of graph is universal in the real world.
For instance, in the social network of Twitter, only a few users (e.g., famous actors) have a large number of followers and their tweets have strong influence, while most users do not~\cite{hafiene2020influential}.
A more typical example is the gene-regulatory networks, which are generally thought to approximate a hierarchical scale-free network topology~\cite{barabasi2004network}, i.e., the number of edges follows a power-law distribution. In addition, one can refer to \cite{albert2005scale,gao2018modeling,himelboim2017classifying} for more examples. It is without doubt that the graphical model with hubs is of significant importance in graph signal processing due to their pivotal role in information flow, community structure, centrality analysis, resource allocation, and dynamic analysis within networks \cite{DongTFV,DongTTBF,OrtegaFKMV,ShumanNFOV}.

\begin{figure}[h]
	\begin{center}
		\subfigure[ ]{
			\includegraphics[width=.35\columnwidth]{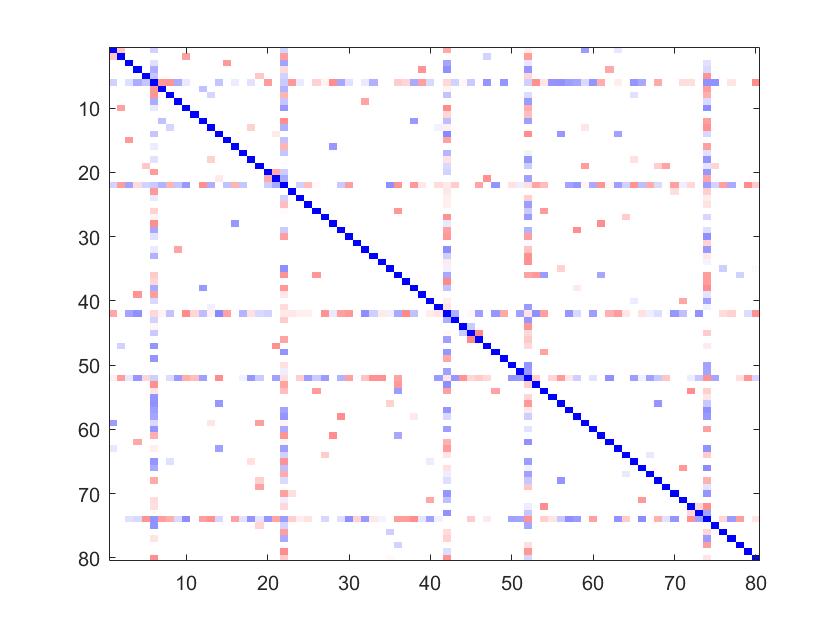}}
		\subfigure[ ]{
			\includegraphics[width=.35\columnwidth]{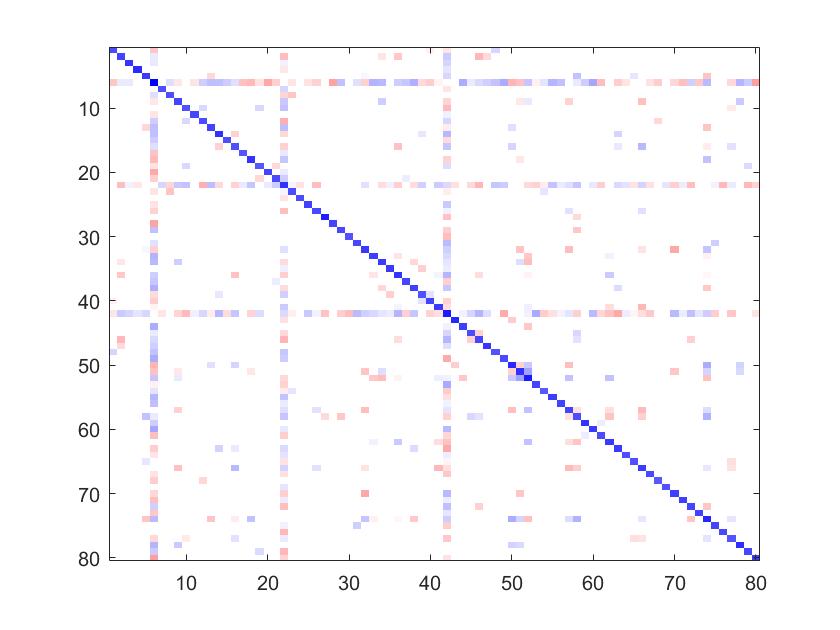}}
		\subfigure[ ]{
			\includegraphics[width=.35\columnwidth]{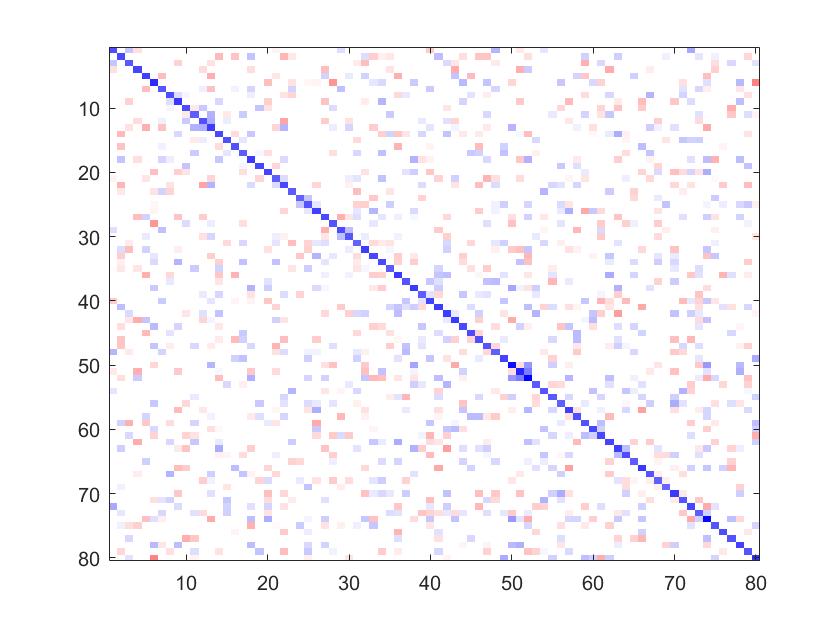}}
		\caption{(a) The inverse covariance matrix in a toy example of a Gaussian graphical model featuring five hub nodes, the inverse covariance matrix displays white elements as zeros and colored elements as non-zeros. Consequently, the colored elements represent the edges in the graph. (b) Estimate from the hub graphical lasso. (c) Graphical lasso estimate
}
		\label{fig:hub_exp}
	\end{center}
\end{figure}

To model a graph with hubs, Tan et al. \cite{tan2014learning} proposed to estimate the precision matrix via adding a sparse-group lasso-type regularization to the Gaussian negative log-likelihood.
Specifically, let $S$ be an empirical covariance matrix with its $j$-th column as $S_{j}$, denote the set of $p \times p$ square matrices by $\mathbb{M}^p$ and the set of $p \times p$ symmetric matrices by $\mathbb{S}^p$.
The problem of learning a graph with hub nodes can be formulated as the following \emph{hub graphical lasso} (HGL) optimization problem:
\begin{equation} \label{eq:HGL}
{\min_{\Theta \in \mathbb{S}^p}} \left\{- \log \det(\Theta) +  \langle S, \Theta \rangle  + P(\Theta) \colon \Theta \succ 0 \right\},
\end{equation}
where $\langle S, \Theta \rangle$ is the standard matrix inner product between $S$ and $\Theta$, $\Theta \succ 0$ means that $\Theta \in \mathbb{S}^p$ is positive definite containing the parameters of interest.
The function $P(\Theta)$ is referred to as the \emph{hub penalty}, which takes the form
\begin{equation*}
\begin{aligned}
P(\Theta)  = & \min_{Z\in \mathbb{S}^p, V \in \mathbb{M}^p}  \Big\{\lambda_1 \|Z - \diag(Z)\|_1 + \lambda_2 \|V - \diag(V)\|_1 \\
& + \lambda_3 \sum_{j=1}^p \|{\left(V - \diag(V) \right)}_j \| \,\Big|\, \Theta = Z+V+V^T \Big\}.
\end{aligned}
\end{equation*}
For any $x\in \mathbb{R}^p$, we denote its $q$-norm as $\|x\|_q = (\sum_{i=1}^p |x_i|^q)^{1/q}$. For simplicity, we denote $\|\cdot\|=\|\cdot\|_2$. For any $X\in \mathbb{M}^p$, we denote $\|X\|_1 = \sum_{i=1}^p\sum_{j=1}^p |X_{ij}|$. This penalty function attempts to decompose the estimated precision matrix $\Theta$ into three matrices $Z$, $V$ and $V^{T}$.
The non-zero entries of $Z$ indicate the edges between the non-hub nodes, and the non-zero columns of $V$ corresponds to the hub nodes.
The $\ell_1$ penalty on the off-diagonal entries of $Z$ promotes the sparsity of the solution, which reduces the links between the non-hub nodes.
The combination of the $\ell_1$ and $\ell_2$ penalties for the columns of $V$ induces group sparsity, which enforces the entries of each column to be either all zeros or almost entirely nonzeros. One may see an example of a network containing hub nodes in Figure \ref{fig:hub_exp}. We observe that when the actual network includes hub nodes (Figure 1(a)), the hub graphical lasso (Figure 1(b)) significantly outperforms the graphical lasso (Figure 1(c)), which is a well-established model that applies an $l_1$ penalty to each edge in the graph, in recovering the network.

However, in the HGL model \eqref{eq:HGL}, no prior knowledge of the hubs is assumed. Consequently, the number of hubs identified is always less than or equal to the true number of hubs. To further make use of the prior knowledge of the hubs that may be available in advance, Li, Bai and Zhou \cite{li2018learning} improved the HGL via dividing the hub penalty in a discriminatory way:
\begin{align*}
\widetilde{P}(\Theta)&= \min_{Z \in \mathbb{S}^p, V \in \mathbb{M}^p} \Big\{  \lambda_1 \|Z-\diag(Z)\|_1 + \lambda_2 \sum_{j \notin \mathcal{D}}\|{(V - \diag(V))}_j\|_1\\
&\quad+ \lambda_3 \sum_{j \notin \mathcal{D}}\|{(V - \diag(V))}_j \|+ \lambda_4 \sum_{j \in \mathcal{D}}\|{(V - \diag(V))}_j\|_1\\
&\quad+ \lambda_5 \sum_{j \in \mathcal{D}} \|{(V - \diag(V))}_j \|\,\Big|\, \Theta = Z+V+V^T \Big\},
\end{align*}
where $\mathcal{D} \subseteq \{1,\dots,p\}$ is the index set of hubs with prior information.
Here we set $\lambda_2 \geq \lambda_4$ and $\lambda_3 \geq \lambda_5$ such that the known hubs are imposed a smaller penalty.
It should be noted that the \emph{discriminated hub graphical lasso} (DHGL) reduces to the HGL if $\mathcal{D} = \emptyset$, and reduces to the classical \emph{graphical lasso} (GL) as $\lambda_2, \lambda_3, \lambda_4, \lambda_5 \to \infty$~\cite{yuan2007model}.

Now we directly focus on the DHGL optimization problem.
For convenience, we use $Q(Z)$ and $R(V)$ to denote the penalty function of $Z$ and $V$, i.e.,
\begin{equation*}
Q(Z) = \lambda_1 \|Z - \diag(Z) \|_1, \quad \forall\, Z\in \mathbb{S}^p,
\end{equation*}
and
\begin{align*}
 R(V) &= \lambda_2 \sum_{j \not\in \mathcal{D}}\|{(V-\diag(V))}_j\|_1 + \lambda_3 \sum_{j \not\in \mathcal{D}} \|{(V-\diag(V))}_j\|  \\
&+ \lambda_4 \sum_{j \in \mathcal{D}}\|{(V-\diag(V))}_j\|_1 +\lambda_5 \sum_{j \in \mathcal{D}} \|{(V-\diag(V))}_j\|,\quad \forall\, V\in \mathbb{M}^p.
\end{align*}
In addition, we define that $\log 0 := -\infty$.
Then, the optimization problem can be written as
\begin{align} \label{P} 
\min_{{\Theta, Z \in \mathbb{S}^p}\atop{V \in \mathbb{M}^p}} & \Big\{ \langle S, \Theta \rangle - \log \det(\Theta) + Q(Z) + R(V) \,\Big|\,\Theta = Z + \mathcal{T}V, \Theta \succeq 0 \Big\},
\end{align}
where $\Theta \succeq 0$ indicates that $\Theta \in \mathbb{S}^p$ is positive semidefinite, $\Theta$ contains the variables of interest, and $\mathcal{T}\colon \mathbb{M}^{p} \to \mathbb{S}^p$ is a linear operator which is defined as $\mathcal{T}V := V+V^T$, for any $V \in \mathbb{M}^{p}$. Although problem \eqref{P} introduces two additional parameters compared with HGL, the selection cost remains manageable due to the structured two-stage cross-validation approach outlined in Algorithms 4 and 5. By tuning parameters in stages rather than simultaneously, the computational burden is kept acceptable.

To solve the optimization problem HGL or DHGL, Tan et al. \cite{tan2014learning} applied the \emph{alternating direction method of multipliers} (ADMM) to the problem~(\ref{eq:HGL}), and Li, Bai and Zhou \cite{li2018learning} applied the ADMM to solve the problem~(\ref{P}). (For short, we call this kind of framework as pADMM.)
Unfortunately, as the scale of the problems becomes large, the efficiency of the pADMM decreases.
On one hand, the ADMM often cannot achieve a solution with high accuracy, which may ruin the performance of the HGL or DHGL model.
On the other hand, the computational cost of the ADMM is extremely high and nearly cannot be acceptable when the dimension $p$ of the data is very high.
Hence, an efficient algorithm to solve the problem~(\ref{P}) accurately is needed.
For the classical graphical lasso problem (where $P(\Theta) = \lambda\|\Theta - \diag{\Theta}\|_1$), a number of optimization algorithms were adopted to solve it, e.g., the ADMM~\cite{scheinberg2010sparse}, Nesterov's smooth gradient method \cite{d2008first,lu2009smooth}, the block coordinate ascent method~\cite{banerjee2008model,friedman2008sparse}, the interior point method~\cite{li2010inexact,yuan2007model}, the Newton method~\cite{hsieh2013sparse,oztoprak2012newton}, and the dual spectral projected gradient method \cite{Nakagaki2020dual}.
Besides, many efficient algorithms have been presented towards other forms of penalty functions.
For example, Hsieh et al. \cite{JMLR:v15:hsieh14a} employed an iterative quadratic approximation to estimate the precision matrix of the reweighted graphical lasso model.
And the similar idea was also taken by \cite{yang2015fused} when considering the fused multiple graphical lasso problem.
In addition, Bybee and Atchad\'e \cite{JMLR:v19:17-218} introduced an approximate majorize-minimize algorithm to compute the Gaussian graphical model with change points. Tarzanagh and Michailidis
\cite{tarzanagh2018estimation} applied a linearized multi-block ADMM to learn the Gaussian graphical model with structured sparse components.
Nevertheless, none of these well-designed algorithms can be directly used to solve the DHGL problem.
Moreover, in order to solve this kind of large-scale problems, the sparsity structure should be further explored delicately.

In this paper, we develop a two-phase algorithm to compute a high-quality solution of the DHGL problem efficiently.
Specifically, in the first phase,
we apply the ADMM to the dual of the problem~(\ref{P}) (for short, we call it dADMM), to generate a good initial point to warm start the second phase. (The implementation of the first phase is necessary because the algorithm in the second phase is costly and it is more meaningful when the iteration point is near the solution.)
In the second phase, we apply the augmented Lagrangian method (ALM) with the subproblems solved by the semismooth Newton (SSN) method to obtain a highly accurate solution. When using the SSN method, we need to thoroughly exploit the sparsity structure of the generalized Jacobian to save the computational cost.
The systematic empirical evaluations show that the algorithm is promising and it outperforms the existing pADMM by a large margin in all data settings. While similar approaches have been applied to various problems, such as semidefinite programming with bound constraints \cite{YangST}, convex quadratic programming \cite{liang2022qppal}, and Gaussian graphical models with hidden clustering structures \cite{LinSTW}, our problem remains particularly challenging due to the complex structure of the regularization function. Our contribution lies not only in leveraging this two-phase framework, but also in conducting substantial nonsmooth analysis to ensure the well-definedness of the approach, the validity of key assumptions, and the efficient implementation of the method.

In summary, the contributions of this paper are twofold:
\begin{itemize}
\item[(a)] We design a highly efficient two-phase algorithm to solve the discriminated hub graphical lasso problem.
In particular, we design a surrogate generalized Jacobian of the proximal mapping of the hub penalty, which paves the way for the SSN method to efficiently solve the inner subproblems of the ALM. 
\item[(b)] We analyze the convergence of the dADMM, the ALM, and the SSN method.
We further conduct systematic experiments to evaluate the proposed algorithm empirically, and the experimental results on both synthetic and real data confirm the general efficacy and the efficiency of our approach.
\end{itemize}

The rest of this paper is organized as follows.
In Section \ref{sec:Preliminaries}, we present the dual of the problem~(\ref{P}) and some necessary preliminaries.
In Section \ref{sec:A two-phase algorithm}, we devote to present the two-phase algorithm.
In Section \ref{sec:Numerical issues}, we describe the numerical issues for solving the involved subproblems.
In Section \ref{sec:Numerical experiments}, we implement the numerical experiments to demonstrate the efficacy and efficiency of the proposed algorithm.
In Section \ref{sec:Conclusion}, we conclude this paper.

\subsection{Additional notations}
\label{subsec:Additional notations}
Let $\mathcal{X}$, $\mathcal{Y}$ and $\mathcal{Z}$ be real finite dimensional Euclidean spaces. We denote $x\circ y$ as the Hadamard product of $x,y\in\mathcal{X}$. Let $f:\mathcal{X}\rightarrow\mathcal{Y}$ and $g:\mathcal{Y}\rightarrow\mathcal{Z}$, then $g\circ f$ denotes the composition of the two functions. Given $r>0$, we denote $\mathcal{B}_{2}^{r}:=\{x\in \mathbb{R}^p\, |\, \|x\|\leq r\}$ and $\mathcal{B}_{\infty}^{r}:=\{x\in \mathbb{R}^p\, |\, \|x\|_{\infty}\leq r\}$. For a given closed convex set $C$ and a vector $x$, we denote the Euclidean projection of $x$ onto $C$ by $\operatorname{\Pi}_{C}(x) := \operatorname{\mbox{argmin}}_{y\in C}{\|x-y\|}$. $I_{p}$ denotes a $p\times p$ identity matrix.

\section{Preliminaries}
\label{sec:Preliminaries}
In this section, firstly, we formulate the dual problem of the DHGL problem.
Then we briefly introduce the proximal mapping and the Moreau envelope, which play important roles in the ALM.
Finally, we calculate the proximal mapping and Moreau envelope for the negative log-determinant function $-\log\det(\cdot)$, the regularization terms $Q(\cdot)$ and $R(\cdot)$ in the DHGL problem.

\subsection{Duality and the optimality condition}

The Lagrangian function associated with the problem~(\ref{P}) is
\begin{equation*}
\begin{aligned}
L(\Theta,Z, V; Y) &:=  \left \langle S, \Theta \right\rangle -  \log \det(\Theta) + Q(Z) + R(V)  + \delta_{\mathbb{S}_+^p}(\Theta)\\ &\quad-\langle Y, Z+\mathcal{T}V - \Theta \rangle, \quad
\forall\, \left(\Theta, Z, V, Y \right) \in \mathbb{S}^p \times \mathbb{S}^p \times \mathbb{M}^{p} \times \mathbb{S}^{p},&
\end{aligned}
\end{equation*}
where $\delta_{\mathbb{S}^p_{+}}(\cdot)$ is an indicator function whose effective domain consists of \textit{p}-order positive semidefinite matrices and $Y\in \mathbb{S}^p$ is the Lagrangian multiplier. For the general theory of the Lagrangian function, one may refer to Section 28 of \cite{Rockafellar1970Convex}.
Furthermore, we can obtain the corresponding dual problem:
\begin{equation} \tag{D} \label{D}
\begin{array}{cl}
{\underset{Y, \Omega,\Lambda \in \mathbb{S}^p} {\max}} & {\log\det(\Omega) - Q^*(Y)-R^*(\Lambda) + p} \\
{\text{s.t.}} & {S-\Omega+Y = 0,} \\
 & {\Lambda - \mathcal{T}^* Y = 0,} \\
 & {\Omega \succeq 0,}
\end{array}
\end{equation}
where $\mathcal{T}^*\colon \mathbb{S}^{p} \to \mathbb{M}^p$ is the adjoint operator of $\mathcal{T}$, $Q^*(\cdot)$ and $R^*(\cdot)$ are the Fenchel conjugate functions of $Q(\cdot)$ and $R(\cdot)$, defined by
\begin{equation*}
\begin{array}{l}
Q^*(Y) = \sup \left\{\left\langle Z, Y \right\rangle - Q(Z) \mid Z \in \mathbb{S}^{p} \right\}, \quad \forall \, {Y}\in \mathbb{S}^{p}, \\
R^*(\Lambda) = \sup \left\{\langle V, \Lambda \rangle - R(V) \mid V \in \mathbb{M}^{p} \right\}, \quad \forall \, \Lambda \in \mathbb{M}^{p},
\end{array}
\end{equation*}
where, $\sup$ denotes the supreme.
Note that for any $Y\in \mathbb{S}^p$, we have $\mathcal{T}^* Y = 2Y$. The KKT condition for the problems (\ref{P}) and (\ref{D}) is
\begin{equation} \label{the KKT conditions}
\left\{
\begin{array}{l}
{Z+\mathcal{T}V = \Theta,} \\
{S-\Omega+Y = 0,} \\
{\Theta\Omega = I_{p}}, \\
{Y \in \partial Q(Z),} \\
{\mathcal{T}^* Y \in \partial R(V),} \\
{\Theta \succeq 0, \quad \Omega \succeq 0.} \\
\end{array}
\right.
\end{equation}

\subsection{The proximal mapping and the Moreau envelope}
In a real finite dimensional Euclidean space $\mathcal{H}$, which is equipped with an inner product $\langle \cdot,\cdot\rangle$ and its induced norm $\|\cdot\|$, given a closed convex function $f:\mathcal{H}\rightarrow \mathbb{R}$, the Moreau envelope of $f$ is defined by
\begin{equation*}
e_f(x) := \min_{y \in \mathcal{H}}\left\{ f(y) + \frac{1}{2}\|y-x\|^2\right\}.
\end{equation*}
The unique minimizer of the above minimization problem is called the proximal mapping of $f$ at $x$, and we denote it by $\textrm{prox}_f(x)$.
That is,
\begin{equation*}
\operatorname{prox}_f(x) := \mathop{\arg\min}_{y \in \mathcal{H}}\left\{f(y) + \frac{1}{2}\|y-x\|^2 \right\}.
\end{equation*}
It is proven in ~\cite[Theorem 2.26]{rockafellar2009variational} that the envelope function $e_f(\cdot)$ is convex and continuously differentiable with
\begin{equation*} \label{gradient of envelope}
\nabla e_f(x) = x - \operatorname{prox}_{f}(x) = \operatorname{prox}_{f^*}(x),
\end{equation*}
where the last equality is derived by the Moreau identity, which states that the following equation holds for any $t > 0$ and $x\in \mathcal{H}$,
\begin{equation*}
\operatorname{prox}_{tf}(x)+t \operatorname{prox}_{f^{*} / t}(x / t)=x.
\end{equation*}

\subsection{Properties of the negative log-determinant function}
For convenience, we denote $h(X) = - \log \det(X)$ for any $X \succeq 0$.
The proximal mapping of $h(\cdot)$ and its corresponding Hessian is well studied in~\cite[Lemma 2.1, Lemma 3.1]{wang2010solving}.
We state them in the following two propositions.

\begin{proposition}
\label{prop:phi1}
For any $X \in \mathbb{S}^p$, let its eigendecomposition be $X = P \textrm{Diag}(d)P^T$, where $d$ is the vector of eigenvalues in the descending order and $P$ is an orthonormal matrix whose columns are the corresponding eigenvectors.
Given $\gamma > 0$, we define two scalar functions
\begin{eqnarray*}
\phi_{\gamma}^+(x):=\frac{\sqrt{x^{2}+4\gamma}+x}{2} &\textrm{ and }& \phi_{\gamma}^-(x):=\frac{\sqrt{x^{2}+4\gamma}-x}{2},\quad\forall\, x\in \mathbb{R}.
\end{eqnarray*}
We also define their vector counterparts $\phi_{\gamma}^+ (\cdot): \mathbb{R}^p \to \mathbb{R}^p$ and $\phi_{\gamma}^- (\cdot): \mathbb{R}^p \to \mathbb{R}^p$ as
\begin{eqnarray*}
\phi_{\gamma}^+(d):=\begin{pmatrix}
                      \phi_{\gamma}^+(d_{1}) \\
                      \vdots \\
                      \phi_{\gamma}^+(d_{p}) \\
                    \end{pmatrix}
 &\textrm{ and }& \phi_{\gamma}^-(x):=\begin{pmatrix}
                      \phi_{\gamma}^-(d_{1}) \\
                      \vdots \\
                      \phi_{\gamma}^-(d_{p}) \\
                    \end{pmatrix},\quad \forall\, d\in \mathbb{R}^{p},
\end{eqnarray*}
and define their matrix counterparts $\phi_{\gamma}^+ (\cdot) \colon \mathbb{S}^p \to \mathbb{S}^p$ and $\phi_{\gamma}^- (\cdot): \mathbb{S}^p \to \mathbb{S}^p$ as
\begin{align*}
&\phi_{\gamma}^+(X):= P\operatorname{Diag}(\phi_{\gamma}^+ (d))P^T,\quad\phi_{\gamma}^-(X):= P\operatorname{Diag}(\phi_{\gamma}^- (d))P^T,\quad
 \forall\, X\in\mathbb{S}^{p}.
\end{align*}
Now we have the following two properties:
\begin{itemize}
\item[(i)] Both of $\phi_{\gamma}^+(X)$ and $\phi_{\gamma}^-(X)$ are positive definite for any $X \in \mathbb{S}^p$, and
\begin{equation*}
X =  \phi_{\gamma}^+(X) - \phi_{\gamma}^-(X), \quad \langle \phi_{\gamma}^+(X) , \phi_{\gamma}^-(X)\rangle = \gamma I_{p}.
\end{equation*}
\item[(ii)] The mapping $\phi_{\gamma}^+(\cdot)$ is continuously differentiable and its Fr\'{e}chet derivative\\ $(\phi_\gamma^+)'(X)$ is
\begin{equation*}
(\phi_{\gamma}^+)'(X)[H] = P(M \circ (P^T HP))P^T, \quad \forall\, H \in \mathbb{S}^p,
\end{equation*}
where $M \in \mathbb{S}^p$ is defined by
\begin{equation*}
M_{ij} = \frac{\phi_{\gamma}^+(d_i) + \phi_{\gamma}^+(d_j)}{\sqrt{d_i^2+4\gamma}+ \sqrt{d_j^2+4\gamma}}, \quad i, j=1,\dots, p.
\end{equation*}
\end{itemize}
\end{proposition}

\begin{proposition}
\label{prop:phi2}
Given $G\in \mathbb{S}^{p}$ and $\sigma>0$, we have
\begin{eqnarray*}
\min_{\Omega \succeq 0}\Big\{\frac{1}{2\sigma}\|G-\Omega\|^{2}-\log\det(\Omega)\Big\}=\frac{1}{2\sigma}\|\phi_{\sigma}^{-}(G)\|^{2}-\log\det(\phi_{\sigma}^{+}(G)),
\end{eqnarray*}
where the minimizer is attained at $\Omega=\phi_{\sigma}^{+}(G)$. That is, ${\rm prox}_{\sigma h}(G) = \phi_{\sigma}^{+}(G)$.
\end{proposition}

\subsection{Properties of the regularization term $Q(\cdot)$}
The regularization term $Q(Z)$ is called the graphical lasso regularizer~\cite{friedman2008sparse}:
\begin{equation*}
Q(Z) = \lambda_1 \|Z-\diag(Z) \|_1 = \sum_{j=1}^p \lambda_1 \| {\left(Z - \diag(Z)\right)}_j \|_1.
\end{equation*}
For $j=1,\dots,p$, let $Z_{[j]} := [Z_{1j}, \dots, Z_{j-1,j}, 0, Z_{j+1,j}, \dots, Z_{p,j}]^{T} \in \mathbb{R}^p$,
and define $\mathcal{P}_j: \mathbb{M}^{p} \to \mathbb{R}^p$ as a linear map satisfying $\mathcal{P}_j Z = Z_{[j]}$. For $x\in\mathbb{R}^{p}$,
the adjoint $\mathcal{P}_j^*: \mathbb{R}^{p} \to \mathbb{M}^{p}$ at $x$ is defined as placing $x$ in the \textit{j}-th column of a $p \times p$ all-zero matrix $X$, and then setting $X_{jj}$ to be zero.
Additionally, define an operator $\mathcal{P}: \mathbb{M}^{p} \to \mathbb{R}^{p^{2}}$ as $\mathcal{P}:= [\mathcal{P}_1; \mathcal{P}_2; \dots; \mathcal{P}_p]$, such that $\mathcal{P}X$ conducts the operation $\mathcal{P}_j$ on the \textit{j}-th column of $X$ for all $j=1,\dots,p$, which actually sets the diagonal elements of $X$ to be zeros and then stacks each column of the corresponding matrix to be a column vercor.
Thus $Q(Z) = \sum_{j=1}^p \lambda_1 \|\mathcal{P}_j Z\|_1$. Then we can compute the closed-form expression of the proximal mapping $\textrm{prox}_Q(\cdot)$, which is stated in the following proposition. We ignore its proof since it is very similar to that of Proposition \ref{The composite of proximal operator}.
\begin{proposition}
\label{prop:proxl1}
For any $Y \in \mathbb{S}^{p}$, the proximal operator of $Q(\cdot)$ at $Y$ can be computed as
\begin{align*}
\operatorname{prox}_Q(Y) &= \mathop{\arg\min}_{Z \in \mathbb{S}^{p}}  \left\{\frac{1}{2}\|Z-Y\|^2 + Q(Z) \right\}
= Y - \mathcal{P}^* \operatorname{\Pi}_{\mathcal{B}^{p,\lambda_{1}}_{\infty}}(\mathcal{P}Y),
\end{align*}
where $\mathcal{B}^{p,\lambda_{1}}_{\infty} := \underbrace{\mathcal{B}_{\infty}^{\lambda_1} \times \dots \times \mathcal{B}_{\infty}^{\lambda_1}}_{p}$,
and $\operatorname{\Pi}_{\mathcal{B}^{p,\lambda_{1}}_{\infty}}(\cdot): \mathbb{R}^{p^{2}} \to \mathbb{R}^{p^{2}}$ is a block-wise projection onto the $\ell_{\infty}$ norm ball with the radius $\lambda_1$.
Specifically, the \textit{j}-th block of $\operatorname{\Pi}_{\mathcal{B}_{\infty}^{p,\lambda_{1}}}(\mathcal{P}Y)$ is
\begin{align*}
[\operatorname{\Pi}_{\mathcal{B}_{\infty}^{p,\lambda_{1}}}(\mathcal{P}Y)]_{(j-1)p+1:jp}&= \operatorname{\Pi}_{\mathcal{B}_{\infty}^{\lambda_1}}(\mathcal{P}_j Y)\\
&=
\operatorname{sign}(\mathcal{P}_j Y) \circ \min(|\mathcal{P}_j Y|, \lambda_1),\quad j=1,\ldots,p.
\end{align*}

\end{proposition}

Next, for any $Y \in \mathbb{S}^{p}$, we define an alternative for the Clarke generalized Jacobian of $\textrm{prox}_{Q}(\cdot)$ at $Y$, i.e., $\partial \textrm{prox}_{Q}(Y)$, as below
\begin{equation}
\widehat{\partial} \textrm{prox}_{Q}(Y) :=  \{\mathcal{I} - \mathcal{P}^* \Sigma\mathcal{P}\mid \Sigma\in \partial \operatorname{\Pi}_{\mathcal{B}_{\infty}^{p,\lambda_{1}}}(\mathcal{P}Y)\},
\label{eq:partial prox_Q}
\end{equation}
where $\partial \operatorname{\Pi}_{\mathcal{B}_{\infty}^{p,\lambda_{1}}}(\mathcal{P}Y)$ is the Clarke generalized Jacobian of $\operatorname{\Pi}_{\mathcal{B}_{\infty}^{p,\lambda_{1}}}(\cdot)$ at $\mathcal{P}Y$, which takes the form
\begin{eqnarray*}
\partial \operatorname{\Pi}_{\mathcal{B}_{\infty}^{p,\lambda_{1}}}(\mathcal{P}Y) = \left\{ \textrm{Diag}(\Sigma_1, \dots,\Sigma_p) \left| \Sigma_j\in \partial \operatorname{\Pi}_{\mathcal{B}_{\infty}^{\lambda_1}}(\mathcal{P}_j Y),\,
 j = 1,\cdots,p\right.\right\}.
\end{eqnarray*}
Particularly, for any $y \in \mathbb{R}^p$, we have
\begin{align*}
&	\partial \operatorname{\Pi}_{\mathcal{B}_{\infty}^{\lambda_1}}(y) =
	\left\{ {\rm Diag}(u) \,\left|\, u\in \mathbb{R}^p,\, u_i\in \left\{
	\begin{array}{cl}
		\{1\}, & \text{if}~|y_i| < \lambda_1 \\
		\{t \mid 0 \leq t \leq 1\}, & \text{if}~|y_i| = \lambda_1 \\
		\{0\}, & \text{if}~|y_i| > \lambda_1
	\end{array}
	\right.,\quad i = 1,\cdots,p.\right.\right\}
\end{align*}

Note that it has been proven in \cite[Example 2.5]{hiriart1984generalized} that for any $H\in\mathbb{S}^{p}$,
\begin{equation*}
\widehat{\partial} \textrm{prox}_{Q}(Y)[H] = \partial \textrm{prox}_{Q}(Y)[H].
\end{equation*}

\subsection{Properties of the regularization term $R(\cdot)$}

The regularization term $R(\cdot)$ is essentially the weighted group graphical lasso regularizer:
\begin{equation*}
\begin{aligned}
&R(V) = \,  \lambda_2 \sum_{j \notin \mathcal{D}} \|{\left(V-\diag{(V)} \right)}_j \|_1 + \lambda_3 \sum_{j \notin \mathcal{D}} \|{\left(V-\diag{(V)} \right)}_j \|  \\
&\quad+ \lambda_4 \sum_{j \in \mathcal{D}} \|{\left(V-\diag{(V)}\right)}_j \|_1 + \lambda_5 \sum_{j \in \mathcal{D}} \|{\left(V-\diag{(V)}\right)}_j \|, \quad \forall\, V\in \mathbb{M}^p.
\end{aligned}
\end{equation*}
For simplicity, for any $V\in \mathbb{M}^p$, we denote
\begin{align*}
\psi(V) &= \sum_{j=1}^p w_{1,j} \| {\left(V-\diag{(V)}\right)}_j\|_1,\quad
 \varphi(V) = \sum_{j=1}^p w_{2,j} \|{\left(V-\diag{(V)}\right)}_j\|,
\end{align*}
where
\begin{equation*}
w_{1,j} = \left\{
\begin{array}{cl}
\lambda_2, &\ \text{if}~j \notin \mathcal{D}, \\
\lambda_4, &\ \text{if}~j \in \mathcal{D},
\end{array}
\right. \quad
w_{2,j} = \left\{
\begin{array}{cl}
\lambda_3, &\ \text{if}~j \notin \mathcal{D}, \\
\lambda_5, &\ \text{if}~j \in \mathcal{D}.
\end{array}
\right.
\end{equation*}
Then, borrowing the linear operators $\mathcal{P}_j, j=1,\dots,p$, we can write
\begin{equation}
\begin{aligned}
R(V) = \psi(V) + \varphi(V) = \sum_{j=1}^p (w_{1,j} \|\mathcal{P}_j V\|_1 + w_{2,j} \|\mathcal{P}_j V\|).
\end{aligned}
\label{eq:R-fun}
\end{equation}

The next proposition provides a crucial property.
It reveals that the proximal mapping $\textrm{prox}_R(\cdot)$ of $R = \psi + \varphi$ can be decomposed into the composition of the proximal mappings $\textrm{prox}_{\psi}(\cdot)$ and $\textrm{prox}_{\varphi}(\cdot)$.

\begin{proposition}
\label{The composite of proximal operator}
For any $Y \in \mathbb{M}^{p}$ and $\sigma > 0$,
the proximal operator of $R(\cdot)$ at $Y$ is
\begin{equation*}
\begin{aligned}
\operatorname{prox}_{R}(Y) &=& \operatorname{prox}_\varphi \circ \operatorname{prox}_\psi (Y)=X - \mathcal{P}^* \operatorname{\Pi}_{\mathcal{B}^{w_2}_2}(\mathcal{P}X),
\end{aligned}
\end{equation*}
where
\begin{equation*}
\begin{aligned}
& X = \operatorname{prox}_{\psi}(Y) =  Y - \mathcal{P}^*\operatorname{\Pi}_{\mathcal{B}^{w_1}_\infty}(\mathcal{P}Y), \\
& \mathcal{B}^{w_2}_2 := \mathcal{B}_2^{w_{2,1}}\times \dots \times \mathcal{B}_2^{w_{2,p}}, \quad \mathcal{B}^{w_1}_\infty := \mathcal{B}_\infty^{w_{1,1}}\times \dots \times \mathcal{B}_\infty^{w_{1,p}},
\end{aligned}
\end{equation*}
and $\operatorname{\Pi}_{\mathcal{B}^{w_1}_{\infty}}(\cdot)$ and $\operatorname{\Pi}_{\mathcal{B}^{w_2}_2}(\cdot)$ are block-wise projections onto $\ell_{\infty}$ and $\ell_{2}$ norm ball, respectively.
Specifically, for $j=1,\ldots,p$, the \textit{j}-th block of $\operatorname{\Pi}_{\mathcal{B}^{w_1}_{\infty}}(\mathcal{P}Y)$ and $\operatorname{\Pi}_{\mathcal{B}^{w_2}_2}(\mathcal{P} X)$ are
\begin{eqnarray*}
[\operatorname{\Pi}_{\mathcal{B}^{w_1}_{\infty}}(\mathcal{P} Y)]_{(j-1)p+1:jp} &=& \operatorname{\Pi}_{\mathcal{B}_{\infty}^{w_{1,j}}}(\mathcal{P}_j Y)\\
&=& \operatorname{sign}(\mathcal{P}_j Y) \circ \min(|\mathcal{P}_j Y|, w_{1,j}),
\end{eqnarray*}
and
\begin{gather} \nonumber
\begin{align}
[\operatorname{\Pi}_{\mathcal{B}^{w_2}_{2}}(\mathcal{P} X)]_{(j-1)p+1:jp} &= \operatorname{\Pi}_{\mathcal{B}_2^{w_{2,j}}}(\mathcal{P}_j X)\\
& =\left\{
\begin{array}{cl}
w_{2,j} \frac{\mathcal{P}_j X}{\|\mathcal{P}_j X\|}, &\quad \text{if}~ \|\mathcal{P}_j X\| >  w_{2,j}, \\
\mathcal{P}_j X, &\quad \text{otherwise,}
\end{array}
\right.
\end{align}
\end{gather}
respectively.
\end{proposition}
\begin{proof}
See Appendix \ref{proof of the composite of proximal operator}.
\end{proof}

Note that even though the proximal mapping $\textrm{prox}_R(\cdot)$ can be obtained by the composition of $\textrm{prox}_{\psi}(\cdot)$ and $\textrm{prox}_{\varphi}(\cdot)$,
the classical chain rule usually cannot work in calculating the generalized Jacobian of $\textrm{prox}_R(\cdot)$ since $\textrm{prox}_{\psi}(\cdot)$ and $\textrm{prox}_{\varphi}(\cdot)$ are non-differentiable \cite{Clarke1983}.
Therefore, we define a surrogate generalized Jacobian of $\textrm{prox}_R(\cdot)$ as follows.
\begin{definition}
For any $Y \in \mathbb{M}^{p}$, $\widehat{\partial}\operatorname{prox}_{R}(Y): \mathbb{M}^{p}\rightrightarrows\mathbb{M}^{p}$ is defined as
\begin{align*}
	\widehat{\partial}\operatorname{prox}_{R}(Y) &:= \left\{ (\mathcal{I} - \mathcal{P}^* \Delta \mathcal{P} ) (\mathcal{I} - \mathcal{P}^* \Xi \mathcal{P} ) \,\left|\,\Delta \in \partial\operatorname{\Pi}_{\mathcal{B}^{w_2}_{2}}(\mathcal{P}X),\, \Xi \in \partial\operatorname{\Pi}_{\mathcal{B}^{w_1}_{\infty}}(\mathcal{P}Y)\right.\right\},
\end{align*}
where $X =Y - \mathcal{P}^*\operatorname{\Pi}_{\mathcal{B}^{w_1}_\infty}(\mathcal{P}Y)$, and
\begin{align*}
\partial\operatorname{\Pi}_{\mathcal{B}^{w_1}_{\infty}}(\mathcal{P}Y) &=  \left\{\operatorname{Diag}(\Xi_1, \dots, \Xi_p) \,\left|\, \Xi_j \in \partial \operatorname{\Pi}_{\mathcal{B}_{\infty}^{w_{1,j}}}(\mathcal{P}_j Y),\,j=1,\cdots,p \right.\right\}, \\
\partial\operatorname{\Pi}_{\mathcal{B}^{w_2}_{2}}(\mathcal{P}X) &= \left\{ \operatorname{Diag}(\Delta_1, \dots, \Delta_p)\,\left|\, \Delta_j \in \partial \operatorname{\Pi}_{\mathcal{B}_{2}^{w_{2,j}}}(\mathcal{P}_j X),\,j=1,\cdots,p\right.\right\}.
\end{align*}
For any $y \in \mathbb{R}^p$ and $j=1,\ldots,p$, we have
\begin{align*}
\partial \operatorname{\Pi}_{\mathcal{B}_\infty^{w_{1,j}}}(y) &= \left\{ {\rm Diag}(u)\,\left| u \in \mathbb{R}^p,\,u_i \in
\left\{
\begin{array}{cl}
\{1\}, & \text{if}\ |y_i| < w_{1,j} \\
\{t \mid 0 \leq t \leq 1\}, & \text{if}\ |y_i| = w_{1,j} \\
\{0\}, & \text{if}\ |y_i| > w_{1,j}
\end{array},\,i=1,\ldots,p.
\right. \right.\right\},\\
\partial \operatorname{\Pi}_{\mathcal{B}^{w_{2,j}}_2}(y) &=
\left\{
\begin{array}{cl}
 \left\{ \frac{w_{2,j}}{\|y\|}(I_{p} - \frac{yy^{T}}{\|y\|^2})\right\}, & \text{if}\ \|y\| >  w_{2,j} \\
\Big\{I_{p} - t \frac{yy^{T}}{(w_{2,j})^2} \Big| 0 \leq t \leq 1\Big\}, & \text{if}\ \|y\| = w_{2,j} \\
\{I_{p}\},  & \text{if}\ \|y\| <  w_{2,j}
\end{array}
\right. .
\end{align*}
\end{definition}

Similarly, by \cite[Example 2.5]{hiriart1984generalized} we have that for any $H\in \mathbb{M}^{p}$, $\widehat{\partial}\textrm{prox}_{R}(Y)[H]=\partial\textrm{prox}_{R}(Y)[H]$.

\begin{proposition}\label{prop:strongly-semismooth}
$\operatorname{prox}_{Q} (\cdot)$ and $\operatorname{prox}_{R} (\cdot)$ are strongly semismooth. That is, let $X,Y\in\mathbb{S}^{p}$, we have
\begin{align}
\operatorname{prox}_{Q} (Y) - \operatorname{prox}_{Q} (X) - \mathcal{W}^{Q}(Y-X) = O(\|Y-X\|^2),\,
 \forall\, \mathcal{W}^{Q} \in \widehat{\partial}\operatorname{prox}_{Q}(Y),
\label{eq:ProxQ_strongly-semismooth}
\end{align}
and for any $U,V\in\mathbb{M}^{p}$, we have
\begin{align}
\operatorname{prox}_{R} (V) - \operatorname{prox}_{R} (U) - \mathcal{W}^{R}(V-U) = O(\|V-U\|^2),\,
 \forall\, \mathcal{W}^{R} \in \widehat{\partial}\operatorname{prox}_{R}(V).
\label{eq:ProxR_strongly-semismooth}
\end{align}
\end{proposition}
\begin{proof}
Since $\operatorname{prox}_{Q} (\cdot)$ is piecewise affine, according to \cite[Proposition 7.4.7]{FacchineiP}, we obtain \eqref{eq:ProxQ_strongly-semismooth}. Similar to \cite[Theorem 3.1]{zhang2020efficient}, we can easily obtain that \eqref{eq:ProxR_strongly-semismooth} holds.
\end{proof}

\section{A two-phase algorithm} \label{sec:A two-phase algorithm}

In this section, we propose a two-phase algorithm to solve the DHGL problem.
In Phase I, we adopt the dADMM to generate a nice initial point to warm start the Phase I\!I algorithm.
It is worth noting that though the dADMM can be used alone to solve the problem, it is not efficient enough to achieve an accurate solution needed in practical applications.
In Phase I\!I, we apply the superlinearly convergent ALM 
to compute a more accurate solution.
Furthermore, in order to fully leverage the sparsity structure of the Jacobians, we use the SSN method to solve the subproblems.

\subsection{Phase I: dADMM}
The minimization form of the dual problem (\ref{D}) is
\begin{equation}
\label{eq:dual-problem}
\begin{array}{cl}
\underset{Y, \Omega, \Gamma, \Lambda\in \mathbb{S}^p}{\min} & -\log\det(\Omega) + Q^*(\Gamma) + R^*(\Lambda) + \delta_{\mathbb{S}^p_+}(\Omega) - p \\
\text{s.t.} & S-\Omega+Y = 0, \\
 & {\Gamma - Y = 0}, \\
 & {\Lambda - \mathcal{T}^* Y = 0}.
\end{array}
\end{equation}
Its corresponding Lagrangian function is given by
\begin{align*} \label{Lagrangian function for dual problem}
\mathcal{L}(Y, \Omega, \Gamma, \Lambda; \Theta, Z, V) := - \log\det(\Omega) + Q^*(\Gamma) +R^*(\Lambda) + \delta_{\mathbb{S}_+^p}(\Omega) - p
 &\\-\left\langle \Theta, S-\Omega+Y \right\rangle - \langle Z, \Gamma - Y \rangle- \langle V, \Lambda - \mathcal{T}^* Y \rangle,&\\
\forall \, \Omega, Y, \Theta, Z,\Gamma,\Lambda \in \mathbb{S}^p, V \in \mathbb{M}^{p}.&
\end{align*}
For $\sigma > 0$, the associated augmented Lagrangian function is
\begin{eqnarray}
\mathcal{L}_{\sigma}(Y, \Omega, \Gamma, \Lambda; \Theta, Z, V)
= \mathcal{L}(Y, \Omega, \Gamma, \Lambda; \Theta, Z, V)+ \frac{\sigma}{2} \|S-\Omega+Y\|^2&\nonumber\\
+\frac{\sigma}{2} \|\Gamma - Y\|^2 + \frac{\sigma}{2} \|\Lambda - \mathcal{T}^* Y\|^2.&\label{augmented Lagrangian function for dual problem}
\end{eqnarray}
Based on the augmented Lagrangian function (\ref{augmented Lagrangian function for dual problem}), we can adopt the dADMM. We summarize the dADMM for solving the problem (\ref{D}) in Algorithm~\ref{alg:1}. One may refer to Section~\ref{sec:Preliminaries} for the detailed computations of the related proximal mappings.

\begin{algorithm}[H]
\renewcommand{\algorithmicrequire}{\textbf{Input:}}
\renewcommand{\algorithmicensure}{\textbf{Output:}}
\caption{dADMM for DHGL}
\label{alg:1}
\begin{algorithmic}[1]
\REQUIRE $\Omega^0,\Theta^0\in \mathbb{S}_{++}^p, Z^0, \Gamma^0, \Lambda^0 \in \mathbb{S}^p, V^0 \in \mathbb{M}^p,  \sigma > 0$, $\tau\in (0,(1+\sqrt{2})/2)$, (a typical choice is $\tau=1.618$). Let $k=0$, iterate as follows:
\STATE Compute
\begin{equation*}
Y^{k+1} = \frac{1}{6\sigma}(\Theta^k - Z^k - V^k - {V^{k}}^{T}) + \frac{1}{6}(\Gamma^k + 2\Lambda^k - (S-\Omega^k)).
\end{equation*}
\STATE Compute
\begin{equation*}
\left\{
\begin{array}{l}
\Omega^{k+1} = \phi_{1/\sigma}^+(S+Y^{k+1} - \Theta^k / \sigma), \\
\Gamma^{k+1} = \textrm{prox}_{Q^*/\sigma}(Y^{k+1} + Z^k / \sigma), \\
\Lambda^{k+1} = \textrm{prox}_{R^*/\sigma}(\mathcal{T}^* Y^{k+1} + V^k / \sigma).
\end{array}
\right.
\end{equation*}
\STATE Compute
\begin{equation*}
\left\{
\begin{array}{l}
\Theta^{k+1} = \Theta^k - \tau \sigma (S-\Omega^{k+1}+Y^{k+1}), \\
Z^{k+1} = Z^k - \tau \sigma (\Gamma^{k+1} - Y^{k+1}), \\
V^{k+1} = V^k - \tau \sigma (\Lambda^{k+1} - \mathcal{T}^* Y^{k+1}).
\end{array}
\right.
\end{equation*}
\STATE $k \leftarrow k+1$, go to Step 1.
\end{algorithmic}
\end{algorithm}
\begin{remark}
The parameter $\sigma$ is theoretically fixed. In practice, setting $\sigma = 1$ already yields satisfactory performance. However, to further enhance the algorithm, an adaptive strategy can be employed, where $\sigma$ is dynamically adjusted based on the ratio of primal infeasibility to dual infeasibility. This adaptive selection helps balance convergence speed of primal and dual solutions.

It is also worth noting that in Algorithm \ref{alg:1}, the minimization of the augmented Lagrangian function involves six matrix variables. In contrast, the pADMM proposed in \cite{li2018learning} requires minimizing the augmented Lagrangian function with nine matrix variables. This not only increases storage requirements but also adds three additional matrix variables, making the problem more complex and difficult to solve efficiently. The higher dimensionality of the augmented Lagrangian function further complicates optimization, potentially leading to smaller iterative steps and slower convergence.
\end{remark}

Now we give the convergence result of the dADMM for solving the DHGL problem.
For clarity and brevity, we simply use $X_P$ and $X_D$ to denote the primal and dual variables, respectively, i.e.,  $X_P :=(\Theta, Z, V) \in  \mathbb{S}^p_+ \times \mathbb{S}^{p}\times \mathbb{M}^{p}$ and $X_D := (Y, \Omega, \Gamma, \Lambda) \in \mathbb{S}^p \times \mathbb{S}^p_+ \times \mathbb{S}^{p} \times \mathbb{S}^{p}$.
For technical reasons, we consider the following constraint qualification (CQ):

\textbf{CQ:} There exists $X_D^{0} := (Y^{0}, \Omega^{0}, \Gamma^{0}, \Lambda^{0}) \in \mathbb{S}^p \times \mathbb{S}^p_{++} \times \mathbb{S}^{p} \times \mathbb{S}^{p}\cap P$, where $P$ is the constraint set in \eqref{eq:dual-problem}.

Based on ~\cite[Theorem B.1]{FazelPSP}, we have the following convergence result.
\begin{theorem}
Assume that the solution set of \eqref{eq:dual-problem} is nonempty and that the \textbf{CQ} holds.
Let $\tau \in (0, (1+\sqrt{5})/2)$.
Let $\{(X_D^k, X_P^k)\}$ be the sequence generated by the dADMM.
Then $\{X_D^k := (Y^k, \Omega^k, \Gamma^k, \Lambda^k)\}$ converges to an optimal solution of the problem (\ref{D}), and $\{X_P^k := (\Theta^k, Z^k, V^k)\}$ converges to an optimal solution of the problem (\ref{P}).
\end{theorem}
\begin{proof}
Under the \textbf{CQ}, it follows from \cite[Corollaries 28.2.2 and 28.3.1]{Rockafellar1970Convex} that $\overline{X}_P:=(\overline{\Theta}, \overline{Z}, \overline{V}) \in  \mathbb{S}^p_+ \times \mathbb{S}^{p}\times \mathbb{M}^{p}$ is an optimal solution to problem \eqref{P} if and only if there exists a Lagrange multiplier $\overline{X}_D:= (\overline{Y}, \overline{\Omega}, \overline{\Gamma}, \overline{\Lambda}) \in \mathbb{S}^p \times \mathbb{S}^p_+ \times \mathbb{S}^{p} \times \mathbb{S}^{p}$ such that
\begin{equation*}
\left\{
\begin{array}{l}
{\overline{Z}+\mathcal{T}\overline{V} = \overline{\Theta},} \\
{\overline{S}-\overline{\Omega}+\overline{Y} = 0,} \\
{\overline{\Theta}\overline{\Omega} = I_{p},} \\
{\overline{Y} \in \partial Q(\overline{Z}),} \\
{\mathcal{T}^* \overline{Y} \in \partial R(\overline{V}),} \\
{\overline{\Theta} \succeq 0, \quad \overline{\Omega} \succeq 0.} \\
\end{array}
\right.
\end{equation*}
Based on the theoretical framework established in \cite[Theorem B.1]{FazelPSP}, we only need to make sure that the assumptions hold.
Since
\begin{equation*}
2\sigma\cI+\sigma\cT\cT^{*}\succ 0,\quad
\left(
             \begin{array}{ccc}
               \Omega^{-1}\otimes\Omega^{-1}+\sigma\cI &  & \\
                & \sigma\cI & \\
                &   & \sigma\cI \\
             \end{array}
           \right)\succ 0,
\end{equation*}
the desired result follows.\qed
\end{proof}

\subsection{Phase I\!I: ALM}
\label{sec:ALM}
For the ALM, we continue to work on the augmented Lagrangian function, where we will update the primal variables simultaneously instead of updating them alternatively as in Phase I. The bottleneck lies in constructing the generalized Jacobian and extracting its underlying structures for efficient computations within the SSN method for solving the ALM subproblems.

At the \textit{k}-th iteration of the ALM for solving the problem (\ref{D}),
the most important task is to solve the following subproblem
\begin{equation} \label{alm_subproblem}
	\min_{Y, \Omega, \Gamma, \Lambda \in \mathbb{S}^p} \left\{\mathcal{L}_{\sigma_k}(Y, \Omega, \Gamma, \Lambda; \Theta^k, Z^k, V^k)\right\},
\end{equation}
where $\sigma_k$ is a given parameter, $(\Theta^k,Z^k,V^k)$ are given points. Since $\Omega, \Gamma, \Lambda$ are independent of each other, we can solve them separately as follows
\begin{equation*}
	\begin{array}{l}
		\Omega' = \phi_{1/\sigma_k}^+(S + Y - \frac{1}{\sigma_k} \Theta^k),  \\
		\Gamma'  = \textrm{prox}_{Q^*/ \sigma_k}(Y + \frac{1}{\sigma_k} Z^k), \\
		\Lambda' = \textrm{prox}_{R^*/\sigma_k}(\mathcal{T}^* Y + \frac{1}{\sigma_k} V^k).
	\end{array}
\end{equation*}
Then we can substitute them to the problem \eqref{alm_subproblem}, we actually solve the subproblem
\begin{eqnarray}
	\label{eq:ALM subproblem2}
	\min_{Y\in \mathbb{S}^p} & \Phi_k(Y),
\end{eqnarray}
where
\begin{align*}
	\Phi_k(Y) &=  \min_{\Omega, \Gamma, \Lambda \in \mathbb{S}^p} \left\{\mathcal{L}_{\sigma_k}(Y, \Omega, \Gamma, \Lambda; \Theta^k, Z^k, V^k) \right\}\\
	&= \sigma_k \left[e_{h/\sigma_k}\left(S+Y - \frac{1}{\sigma_k}\Theta^k\right) + e_{Q^*/\sigma_k}\left(Y + \frac{1}{\sigma_k} Z^k\right)\right.\\
	&\left.\quad +e_{R^*/\sigma_k}\left(\mathcal{T}^* Y + \frac{1}{\sigma_k} V^k\right)\right]-\frac{1}{2\sigma_k} \left(\|\Theta^k\|^2 + \|Z^k\|^2 + \|V^k\|^2 \right)- p.
\end{align*}
For clarity, we present the framework of the ALM in Algorithm~\ref{alg:2}.

\begin{algorithm}[H]
	\renewcommand{\algorithmicrequire}{\textbf{Input:}}
	\renewcommand{\algorithmicensure}{\textbf{Output:}}
	\caption{ALM for DHGL}
	\label{alg:2}
	\begin{algorithmic}[1]
		\REQUIRE $\Theta^0\in \mathbb{S}^p_{++}, Z^0 \in \mathbb{S}^p, V^0 \in \mathbb{M}^p;  \sigma_0 > 0$. Let $k=0$, iterate as follows:
		\STATE Apply Algorithm \ref{alg:3} to compute
		\begin{equation*} \label{subproblem}
			\begin{aligned}
				Y^{k+1} \approx \mathop{\arg\min}_{Y \in \mathbb{S}^p} \Phi_k(Y).
			\end{aligned}
		\end{equation*}
		\STATE Compute
		\begin{equation*}
			\begin{aligned}
				& \Omega^{k+1} = \phi_{1/\sigma_k}^+ \left(S+Y^{k+1}-\frac{1}{\sigma_k}\Theta^k \right), \\
				& \Gamma^{k+1} = \textrm{prox}_{Q^*/\sigma_k} \left(Y^{k+1}+\frac{1}{\sigma_k}Z^k \right), \\
				& \Lambda^{k+1} = \textrm{prox}_{R^*/\sigma_k} \left(\mathcal{T}^* Y^{k+1} + \frac{1}{\sigma_k}V^k \right).
			\end{aligned}
		\end{equation*}
		\STATE Compute
		\begin{equation*}
			\begin{aligned}
				& \Theta^{k+1} = \Theta^k - \sigma_k \big(S-\Omega^{k+1}+Y^{k+1} \big), \\
				& Z^{k+1} = Z^k - \sigma_k \big(\Gamma^{k+1} - Y^{k+1} \big), \\
				& V^{k+1} = V^k - \sigma_k \big(\Lambda^{k+1} - \mathcal{T}^* Y^{k+1} \big).
			\end{aligned}
		\end{equation*}
		\STATE Update $\sigma_{k+1}=2\sigma_{k}$ and $k \leftarrow k+1$, go to Step 1.
	\end{algorithmic}
\end{algorithm}

\begin{remark}
In Step 1 of Algorithm \ref{alg:2}, we apply the SSN method to obtain the approximate solution of the inner subproblem. The details of the implementation, along with the corresponding convergence results, are provided in Section \ref{sec: ssn}. This method ensures a solid foundation for solving the subproblem efficiently.

Building on this, Algorithm \ref{alg:2} utilizes a second-order method to solve the DHGL problem. While the computational and memory costs per iteration are of the same order as the first-order method dADMM in Algorithm 1, the second-order method ALM achieves significantly faster convergence. This improved efficiency makes our algorithm highly effective, achieving better results with lower computational and memory costs compared to existing approaches in \cite{li2018learning,tan2014learning}.
\end{remark}

\subsubsection{An implementable stopping criteria for the ALM subproblems}
Since the subproblem \eqref{alm_subproblem} often has no analytical solution,
the stopping criteria for the subproblem in the algorithm needs to be discussed.
For $X_D := (Y, \Omega, \Gamma, \Lambda) \in \mathbb{S}^p \times \mathbb{S}^p_+ \times \mathbb{S}^{p} \times \mathbb{S}^{p}$ and $X_P :=(\Theta, Z, V) \in  \mathbb{S}^p_+ \times \mathbb{S}^{p}\times \mathbb{M}^{p}$, we define functions $f_k$ and $g_k$ as follows
\begin{equation*}
\begin{aligned}
f_k(X_D) &:= \mathcal{L}_{\sigma_k}(X_D; X_P^k) = \sup_{X_P} \Big\{\mathcal{L}(X_D;  X_P) - \frac{1}{2\sigma_k}\|X_P - X_P^k\|^2 \Big\}, \\
g_k(X_P) &:= \inf_{X_D} \Big\{\mathcal{L}(X_D;  X_P) - \frac{1}{2\sigma_k}\|X_P - X_P^k\|^2 \Big\} \\
& = \log \det(\Theta) -\langle S, \Theta \rangle - Q(Z) - R(V) - \frac{1}{2\sigma_k}(\|\Theta-\Theta^k\|^2\\
&\quad +\|Z-Z^k\|^2 + \|V-V^k\|^2) - \delta_{F_P}(\Theta, Z, V),
\end{aligned}
\end{equation*}
where $F_P := \left\{(\Theta, Z, V) \mid \Theta -  Z  - \mathcal{T}V = 0 \right\}$.

To obtain $X_{D}^{k+1}$, the following criteria on the approximate computation of the subproblem~(\ref{alm_subproblem}) were proposed in~\cite{rockafellar1976augmented}:
\begin{subequations}
\begin{align}
& f_{k} (X_D^{k+1}) - \inf f_{k}(X_D) \leq \frac{\varepsilon_k^2}{2 \sigma_k},\, \varepsilon_{k} \ge 0, \, \sum_{k=0}^{\infty} \varepsilon_k < \infty,  \label{stop1} \tag{A} \\
&  f_{k} (X_D^{k+1}) - \inf f_{k}(X_D) \leq \frac{\delta_k^2}{2\sigma_k}\|X_P^{k+1} - X_P^k\|^2, 0 \leq \delta_k \leq 1,\,\sum_{k=0}^{\infty} \delta_k < \infty. \label{stop2} \tag{B}
\end{align}
\end{subequations}
The convergence results when executing the stopping criteria~(\ref{stop1}) and~(\ref{stop2}) for~(\ref{alm_subproblem}) can be obtained directly by adopting~\cite[Theorem 4, 5]{rockafellar1976augmented}.
However, the stopping criteria~(\ref{stop1}) and~(\ref{stop2}) are usually not in a direct usage due to the unknown value of $\inf f_{k}(X_D)$.
To obtain implementable stopping criteria, we need to find an upper bound for $f_{k} (X_D^{k+1}) - \inf f_{k}(X_D)$.
Since
\begin{equation*}
\inf f_k (X_D) = \sup g_k(X_P) \ge g_k(X_P^{k+1}),
\end{equation*}
we have
\begin{equation*}
f_{k} (X_D^{k+1}) - \inf f_k (X_D) \leq f_{k} (X_D^{k+1}) - g_k(X_P^{k+1}).
\end{equation*}
Hence, we terminate the subproblem \eqref{alm_subproblem} if $(X_D^{k+1}, X_P^{k+1})$ satisfies the following conditions:
given nonnegative summable sequences $\{\varepsilon_k\}$ and $\{\delta_k\}$ such that $\delta_k < 1$ for all $k \ge 0$,
\begin{subequations}
\begin{align}
& f_{k} (X_D^{k+1}) - g_k(X_P^{k+1}) \leq \frac{\varepsilon_k^2}{2 \sigma_k}, \, \varepsilon_{k} \ge 0, \, \sum_{k=0}^{\infty} \varepsilon_k < \infty,  \label{stop3} \tag{A$'$} \\
&  f_{k} (X_D^{k+1}) - g_k(X_P^{k+1}) \leq \frac{\delta_k^2}{2\sigma_k}\|X_P^{k+1} - X_P^{k}\|^2,  \, 0 \leq \delta_k \leq 1,\, \sum_{k=0}^{\infty} \delta_k < \infty. \label{stop4} \tag{B$'$}
\end{align}
\end{subequations}

\subsubsection{Convergence results of the ALM}

The global convergence and asymptotically superlinear convergence of the ALM have been carefully explored by~\cite{rockafellar1976augmented}, thus we can directly obtain the result without much effort.
Firstly, by denoting $f(X_P)$ as the objective function of~(\ref{P})
\begin{equation*}
f(X_P) = \langle S, \Theta \rangle - \log\det(\Theta) + Q(Z) + R(V) + \delta_{F_P}(\Theta, Z, V),
\end{equation*}
we can define the maximal monotone operator $\mathcal{T}_{f} := \partial f$
and its inverse operator
\begin{equation*}
\mathcal{T}^{-1}_{f}(U_P) = \underset{X_P}{\arg\min} \Big\{f_k(X_P) - \langle U_P, X_P \rangle \Big \}.
\end{equation*}
Now we are ready to present the convergence result of the ALM for solving the dual problem.

\begin{theorem}
(i) Let $\{(X_D^k, X_P^k)\}$ be the sequence generated by the ALM with the stopping criterion~(\ref{stop3}).
Then $\{(X_D^k, X_P^k)\}$ is bounded, $\{X_D^k\}$ converges to an optimal solution of (\ref{D}), and $\{X_P^k\}$ converges to an optimal solution of (\ref{P}).

(ii) Let $\rho$ be a positive number
such that $\rho > \sum_{k=0}^{\infty} \varepsilon_k$.
Assume that there exists $\kappa > 0$ such that
\begin{equation*}
\mathrm{dist}\big(X_P, \mathcal{T}_f^{-1}(0)\big) \leq \kappa \mathrm{dist}\big(0, \mathcal{T}_{f}(X_P)\big)
\end{equation*}
for all $X_P$ satisfying $\mathrm{dist}\big(X_P, \mathcal{T}_f^{-1}(0)\big) \leq \rho$.
Suppose that the initial point $X_P^0$ satisfies
\begin{equation*}
\mathrm{dist}\big(X_P^0, \mathcal{T}_f^{-1}(0)\big) \leq \rho - \sum_{k=0}^\infty \varepsilon_k.
\end{equation*}
Let $\{(X_D^k, X_P^k)\}$ be the sequence generated by the ALM with the stopping criteria (\ref{stop3}) and (\ref{stop4}).
Then for $k \ge 0$, it holds that
\begin{equation*}
\mathrm{dist}\big(X_P^{k+1}, \mathcal{T}_f^{-1}(0)\big) \leq \varrho_{k} \mathrm{dist}\big(X_P^{k} , \mathcal{T}_f^{-1}(0)\big),
\end{equation*}
where
\begin{eqnarray*}
\varrho_{k}:=\left[(1+\delta_{k})\kappa\left(\kappa^2+\sigma_{k}^2\right)^{-1 / 2}+\delta_{k}\right] /\left(1-\delta_k \right)\,\longrightarrow \varrho_{\infty}:=\kappa\left(\kappa^{2}+\sigma_{\infty}^{2}\right)^{-1 / 2},&\\ \left(\varrho_{\infty}=0 \text { if } \sigma_{\infty}=\infty\right).&
\end{eqnarray*}
\end{theorem}

\subsubsection{The SSN for solving the ALM subproblem}\label{sec: ssn}

Solving the subproblem \eqref{eq:ALM subproblem2} is the most critical part for the efficiency of the algorithm.
By ignoring the superscript, the subproblem at each iteration can be formulated as
\begin{align}  \label{subproblem wrt Y}
\min_{Y \in \mathbb{S}^p} \Big\{\Phi(Y) :=\sigma \left[e_{h/\sigma}\left(S+Y - \frac{1}{\sigma} \overline{\Theta}\right) + e_{Q^*/\sigma}\left(Y+ \frac{1}{\sigma} \overline{Z}\right)\right.&\nonumber\\
 \left.+ e_{R^*/\sigma}\left(2Y + \frac{1}{\sigma} \overline{V}\right) \right]  -\frac{1}{2\sigma} (\|\overline{\Theta}\|^2 + \| \overline{Z} \|^2+ \| \overline{V}\|^2 ) - p \Big\},&
\end{align}
where $\overline{\Theta}, \overline{Z}, \overline{V}$ are fixed.

Note that $\Phi(\cdot)$ is strictly convex and continuously differentiable,
and for any $Y \in \mathbb{S}^p$,
\begin{equation*}
\begin{aligned}
\nabla \Phi(Y) = & \, -\phi^+_{\sigma}(-\sigma(S+Y)+\overline{\Theta})  + \textrm{prox}_{\sigma Q}(\sigma Y + \overline{Z})  +2\textrm{prox}_{\sigma R} (2\sigma Y + \overline{V}).
\end{aligned}
\end{equation*}
Thus, the unique solution $\overline{Y}$ of \eqref{subproblem wrt Y} can be obtained by solving the following nonsmooth system of equations
\begin{equation*}
\nabla \Phi(Y) = 0.
\end{equation*}
We adopt the SSN method to solve this nonsmooth system. However, it is difficult to characterize the structure of $\partial(\nabla \Phi)(Y)$ exactly.
Therefore, we also define a multifunction $\mathcal{V}$ as a surrogate of $\partial(\nabla \Phi)(Y)$.
The multifunction $\mathcal{V} \colon \mathbb{S}^p \rightrightarrows \mathbb{S}^{p}$ is defined as follows
\begin{equation*}
\begin{aligned}
\mathcal{V}(Y) := & ~\sigma \Big[(\phi^{+}_{\sigma})'(-\sigma(S+Y)+\overline{\Theta}) + \widehat{\partial} \textrm{prox}_{\sigma Q}(\sigma Y + \overline{Z}) + 4\widehat{\partial} \textrm{prox}_{\sigma R}(2\sigma Y+\overline{V}) \Big].
\end{aligned}
\end{equation*}
The SSN method is presented in Algorithm~\ref{alg:3}.
\begin{algorithm}[h]
\renewcommand{\algorithmicrequire}{\textbf{Input:}}
\renewcommand{\algorithmicensure}{\textbf{Output:}}
\caption{SSN for the ALM subproblem}
\label{alg:3}
\begin{algorithmic}[1]
\REQUIRE $\eta \in (0,1), \beta \in (0,1]$, $\mu \in (0, 1/2)$, and $\delta \in (0,1)$. Choose $Y^0 \in \mathbb{S}^p$ and  $j=0$.
\STATE Select $\mathcal{W}^Q\in\widehat{\partial} \textrm{prox}_{\sigma Q}(\sigma Y^{j} + \overline{Z})$, $\mathcal{W}^R\in\widehat{\partial} \textrm{prox}_{\sigma R}(2 \sigma Y^{j}+\overline{V})$. Let $W=\sigma [(\phi^{+}_{\sigma})'(-\sigma(S+Y^j)+\overline{\Theta}) + \mathcal{W}^Q + 4\mathcal{W}^R ]$. Solve the following linear system by the Preconditioned Conjugate Gradient (PCG) method (with the preconditioner from Section \ref{sec:Numerical issues})
\begin{equation} \label{linear systems for CG}
W[H] = - \nabla \Phi(Y^j)
\end{equation}
to find $H^j \in\mathbb{S}^p$ such that $\|W[H^j]+\nabla \Phi(Y^j)\| \leq \min(\eta, \|\nabla \Phi(Y^j)\|^{1+\beta})$.
\STATE Set $\alpha_j = \delta^{m_j}$ , where $m_j$ is the smallest nonnegative integer $m$ for which
\begin{equation*}
\Phi(Y^j + \delta^m H^j) \leq \Phi(Y^j)+\mu \delta^m \langle \nabla \Phi(Y^j), H^j \rangle.
\end{equation*}
\STATE Set $Y^{j+1} = Y^{j} +\alpha_j H^j$.
\STATE $j \leftarrow j+1$, go to Step 1.
\end{algorithmic}
\end{algorithm}

\begin{remark}
The parameter settings in the algorithms are designed to ensure convergence, not for tuning purposes. In practical computation, we set $\eta=10^{-1}$, $\beta=10^{-1}$, $\mu=10^{-3}$ and $\delta=\frac{1}{2}$.
\end{remark}

Similar to the proof of \cite[Proposition 2]{WangT2021}, we can prove that $\nabla \Phi$ is strongly semismooth in combination with Proposition \ref{prop:phi1} and Proposition \ref{prop:strongly-semismooth}, based on which we can get the following convergence result of the SSN method.
\begin{theorem}
Let $\{Y^{j}\}$ be the sequence generated by Algorithm~\ref{alg:3}, then $\{Y^j\}$ converges to the unique optimal solution $\overline{Y}$ of (\ref{subproblem wrt Y}), and the convergence rate is at least superlinear:
\begin{equation*}
\| Y^{j+1} - \overline{Y}\| = \mathcal{O}( \| Y^j - \overline{Y} \|^{1+\beta} ), \quad \beta \in (0, 1].
\end{equation*}
\end{theorem}

\section{Numerical issues for solving the subproblem \eqref{linear systems for CG}} \label{sec:Numerical issues}

The most important part in implementing Algorithm \ref{alg:3} is how to efficiently solve the subproblem \eqref{linear systems for CG}. The Newton system we need to solve is
\begin{eqnarray}
\label{eq:Newton system}
\Big[(\phi^{+}_{\sigma})'(-\sigma(S+Y)+\overline{\Theta})[H] + \mathcal{W}^{Q}(H) + 4\mathcal{W}^{R}(H) \Big] =
-\sigma^{-1}\nabla \Phi(Y),
\end{eqnarray}
where $\mathcal{W}^{Q}\in \widehat{\partial} \textrm{prox}_{\sigma Q}(\sigma Y + \overline{Z})$ and $\mathcal{W}^{R}\in \widehat{\partial} \textrm{prox}_{\sigma R}(2 \sigma Y+\overline{V})$. Here, we can choose $\mathcal{W}^{Q}$ and $\mathcal{W}^{R}$ such that
\begin{eqnarray*}
\mathcal{W}^{Q}(H) = H - U \circ H,
\end{eqnarray*}
where
\begin{eqnarray*}
U_{ij} = \left\{
\begin{array}{cl}
0, &\quad\text{if}~|(\sigma Y + \overline{Z})_{ij}| > \sigma\lambda_1 ~\text{or}~ i=j, \\
1, &\quad \text{otherwise},
\end{array}
\right.
\end{eqnarray*}
and
\begin{eqnarray*}
\mathcal{W}^{R}(H) = H - \Sigma^\psi\circ H - \mathcal{P}^* \Sigma^\varphi \mathcal{P}(H - \Sigma^\psi\circ H),
\end{eqnarray*}
where
\begin{eqnarray*}
\Sigma^\psi_{ij} = \left\{
\begin{array}{cl}
0, &\quad \text{if}~|(2 \sigma Y+\overline{V})_{ij}| > \sigma w_{1, j} ~\text{or}~ i = j, \\
1, &\quad \text{otherwise},
\end{array}
\right. \\
\end{eqnarray*}
and
\begin{eqnarray*}
\Sigma^\varphi = \textrm{Diag}(\Sigma^\varphi_{1},\ldots,\Sigma^\varphi_{p}),
\end{eqnarray*}
for $j=1,\ldots,p$,
\begin{gather} \nonumber
\begin{align}
 & \Sigma^\varphi_{j} =
\left\{
\begin{array}{cl}
\frac{\sigma w_{2,j}}{\|z_{j}\|}(I - \frac{z_{j}z_{j}^{T}}{\|z_{j}\|^2}), &\quad \text{if}~ \|z_{j}\| >  \sigma w_{2,j}, \\
I,  &\quad \text{if}~\|z_{j}\| \leq  \sigma w_{2,j},
\end{array}
\right.
\end{align}
\end{gather}
where $z_{j}=\mathcal{P}_{j}X$ with $X = 2 \sigma Y+\overline{V} - \mathcal{P}^* \operatorname{\Pi}_{\mathcal{B}^{w_1}_{\infty}} (\mathcal{P} (2 \sigma Y+\overline{V}))$.

We need to apply the PCG method to solve the system \eqref{eq:Newton system}. To achieve faster convergence, we hope to devise an easy-to-compute preconditioner. Similar to \cite{wang2010solving}, we define an operator $\mathcal{T}^{\phi}: \mathbb{S}^{p}\rightarrow\mathbb{S}^{p}$ as $\mathcal{T}^{\phi}(H):=P(M\circ(P^{T}HP))P^{T}$. As we know, the standard basis in $\mathbb{S}^{p}$ is given by $\{E_{ij}:=\alpha_{ij}(e_{i}e_{j}^{T}+e_{j}e_{i}^{T}): 1\leq i \leq j \leq p\}$, where $e_{i}$ is the \textit{i}-th unit vector in $\mathbb{R}^{p}$ and $\alpha_{ij}=1/\sqrt{2}$ if $i\neq j$ and $\alpha_{ij}=1/2$ otherwise. Let $\mathbf{T}^{\phi}, \mathbf{W}^{Q}$ and $\mathbf{W}^{R}$ denote the matrix representations of the linear operators $\mathcal{T}^{\phi}, \mathcal{W}^{Q}$ and $\mathcal{W}^{R}$, respectively. Then the diagonal element of $\textbf{T}^{\phi}$ with respect to the basis element $E_{ij}$ is
\begin{align}
  {\bf T}^{\phi}_{(ij),(ij)} =& \langle E_{ij}, \mathcal{T}^{\phi}(E_{ij})\rangle = \langle P^T E_{ij} P,M \circ (P^T E_{ij} P)\rangle \nonumber \\
  =&
  \left\{ \begin{array}{ll}
     ( (P\circ P) M (P\circ P)^T)_{ij} +
        \langle v^{(ij)},M v^{(ij)}\rangle, &\quad\mbox{if $i\not= j$}, \\[5pt]
        ( (P\circ P) M (P\circ P)^T)_{ij}, &\quad\mbox{otherwise},
        \end{array} \right.
        \label{eq-diag-T}
\end{align}
where $v^{(ij)} = P_i \circ P_j$ and $P_i,P_j$ are the $i$th and $j$th rows of $P$, respectively. The diagonal element of $\textbf{W}^{Q}$ with respect to the basis element $E_{ij}$ is
\begin{align}
\label{eq-diag-WQ}
\textbf{W}^{Q}_{(ij),(ij)}&=\langle E_{ij},\mathcal{W}^{Q}(E_{ij})\rangle =\langle E_{ij},E_{ij}-U\circ E_{ij}\rangle
=1-U_{ij}.
\end{align}
The diagonal element of $\textbf{W}^{R}$ with respect to the basis element $E_{ij}$ is
\begin{align}
\label{eq-diag-WR}
\textbf{W}^{R}_{(ij),(ij)} &= \langle E_{ij},\mathcal{W}^{Q}(E_{ij})\rangle \nonumber \\
&= \langle E_{ij},E_{ij}-\Sigma^{\psi}\circ E_{ij}-\mathcal{P}^{*}\Sigma^{\varphi}\mathcal{P}(E_{ij}- \Sigma^{\psi}\circ E_{ij})\rangle \nonumber \\
&= 1-\Sigma^{\psi}_{ij}-\langle E_{ij},\mathcal{P}^{*}\Sigma^{\varphi}\mathcal{P}(E_{ij}-\Sigma^{\psi}\circ E_{ij})\rangle.
\end{align}
In \eqref{eq-diag-T}, we note that to compute all the diagonal elements of $\textbf{T}^{\phi}$, the computational cost of the terms $\{\langle v^{(ij)},M v^{(ij)}\rangle\}_{i,j}$ is $O(p^{4})$ flops. Fortunately, the term $((P\circ P) M (P\circ P)^T)_{ij}$ is a very good approximation of ${\bf T}^{\phi}_{(ij),(ij)}$, whose computational cost is only $O(p^{3})$ flops for all $i,j$. Similarly, in \eqref{eq-diag-WR}, the computational cost of {the terms $\{\langle E_{ij},\mathcal{P}^{*}\Sigma^{\varphi}\mathcal{P}(E_{ij}-\Sigma^{\psi}\circ E_{ij})\rangle \}_{i,j}$ is also $O(p^{4})$ flops, and $1-\Sigma^{\psi}_{ij}$ is a good approximation to $\textbf{W}^{R}_{(ij),(ij)}$. Hence, we propose a preconditioner $\mathcal{V}_{D}:\mathbb{S}^{p}\rightarrow\mathbb{S}^{p}$, such that
\begin{eqnarray*}
\mathcal{V}_{D}(Y)[H]=D\circ H,
\end{eqnarray*}
where $D\in\mathbb{S}^{p}$ and $D_{ij}=((P\circ P) M (P\circ P)^T)_{ij}+1-U_{ij}+4(1-\Sigma^{\psi}_{ij})$, for $1\leq i, j \leq p$.

\section{Numerical experiments} \label{sec:Numerical experiments}
In this section, we demonstrate the performance of our two-phase algorithm.
The pADMM and dADMM serve as benchmark algorithms.
We evaluate the efficacy and efficiency of our algorithm on both synthetic and real-world data.
The algorithms are implemented in {\sc Matlab} R2019a.
All the numerical experiments in this paper are carried out on a laptop computer based on 64-bit $\rm Windows$10 system. The computer is configured as: $\rm Intel(R)$$\rm Core(TM) \ i5-6200U \ CPU\ @2.30GHz $$\rm 2.40GHz$, 4G running memory.

\subsection{Experimental setup} 

\subsubsection{Stopping criterion}
In our experiments, we define the $R_P$, $R_D$, and $R_C$ to quantify the infeasibilities of the primal and dual problems, and the complementarity condition,
where
{\small
\begin{equation*}
\begin{aligned}
R_P &:= \frac{\|\Theta - Z - V - V^T \|}{1+\|\Theta\|}, \quad R_D := \frac{\|S-\Omega+Y\|}{1+\|S\|}, \\
R_C &:= \max \left\{ \frac{\|\Theta\Omega - I\|}{1+\|\Theta\|+\|\Omega\|}, \frac{\|Z - \textrm{prox}_Q(Y+Z)\|}{1+\|Z\|},\,\frac{\|V - \textrm{prox}_R(V + \mathcal{T}^* Y)\|}{1+\|V\|} \right\}.
\end{aligned}
\end{equation*}
}We stop the algorithm when $R_P$, $R_D$, and $R_C$ achieve a sufficient precision:
\begin{equation*}
\max\left\{R_P, R_D, R_C \right\} < \texttt{Tol}=10^{-6}.
\end{equation*}
We also use these three precision metrics to determine when to start Phase I\!I of the algorithm.
Particularly, we switch to phase I\!I if $\max\{R_P, R_D, R_C\} < 10^{-4}$.
Moreover, we fix the maximum iteration numbers of the algorithms.
Specifically,
in our two-phase algorithm, the Phase I or Phase I\!I algorithm is terminated if the number of the iterations reaches 200.
In the benchmark algorithms, we confine the maximum number of iterations to 10,000.

\subsubsection{Performance measures}

Given examples $\mathbf{x}_{1}, \dots, \mathbf{x}_{n} \stackrel{\mathrm{i.i.d.}}{\sim} \mathcal{N}(\mathbf{0}, \Sigma)$ with $p$ features,
the goal of the DHGL is to estimate the inverse of the matrix $\Sigma$.
Let $\widehat{\Theta}$ be the final estimation of the inverse empirical covariance matrix,
we directly regard the entry of $\widehat{\Theta}$ whose value is less than a threshold $\epsilon$ as $0$, such that we can analyze its sparsity pattern.
In our experiments, the threshold $\epsilon$ is fixed to be $10^{-5}$.
We refer to the index set of the estimated hub nodes by $\widehat{\mathcal{H}}_r$,
and the rule of identifying a node as a hub is that it contains more than $r$ edges,
i.e.,
$\widehat{\mathcal{H}}_r = \big\{ i \mid \sum_{j=1, j \neq i}^p (1_{\{ |\Omega_{ij}| > \epsilon \}} ) > r \big\}$.
In our experiments, we simply set $r$ to be $p/5$.

\textbf{(1) On synthetic data}. The true covariance matrix $\Sigma$ is known. 
Thus the true hubs in the graph is available, and we denote the index set of the true hub nodes by $\mathcal{H}$.
Hence, we can use the following four most commonly used measurements to evaluate the efficacy of the algorithms.
\begin{itemize}
\item[(a)] Correctly estimated number of edges:
{\small
\begin{equation*}
\sum_{j<j'}\left(1_{\left\{\left|\Theta_{jj'}\right| > \epsilon ~\mathrm{and}~ \left|\Sigma^{-1}_{jj'}\right| \neq 0 \right\}}\right).
\end{equation*}
}
\item[(b)] Proportion of correctly estimated hub nodes:
{\small
\begin{equation*}
\frac{\left|\mathcal{H}_r \cap \mathcal{H}\right|}{\left|\mathcal{H}\right|}.
\end{equation*}
}
\item[(c)] Proportion of correctly estimated hub edges:
{\small
\begin{equation*}
 \frac{\sum\limits_{j \in \mathcal{H}, j' \neq j} \left(1_{\left\{\left|\Theta_{jj'}\right|>\epsilon~\mathrm{and}~ \left|\Sigma^{-1}_{jj'}\right| \neq 0 \right\}} \right)}{\sum\limits_{j \in \mathcal{H}, j' \neq j} \left(1_{\left\{\left|\Sigma^{-1}_{jj'}\right| \neq 0 \right\}}\right)}.
\end{equation*}
}
\item[(d)] The mean squared error:
\begin{equation*}
\mathcal{M}_{\rm MSE} := \frac{\sum\limits_{j < j'} {\left(\Theta_{jj'} - \Sigma^{-1}_{jj'} \right)}^2}{p(p-1)}.
\end{equation*}
\end{itemize}
\textbf{(2) On real-world data}. In the case of real data, the actual covariance matrix $\Sigma$  is unknown and cannot be evaluated using the previous four measurements, so we only compare the computational time for different algorithms.

\subsubsection{Hyperparameter tuning} \label{sec:hyperparameter tuning}

Similar to other machine learning models, the hyperparameter tuning also plays a critical role in the DHGL.
Tan et al. \cite{tan2014learning} proposed to select the hyperparameters by optimizing a Bayesian information criterion (BIC)-type quantity, but this quantity tends to yield
unnecessarily dense graphs when the dimension $p$ is large~\cite{drton2017structure}, and it also introduces two additional hyperparamters that are too sensitive to the quality of the solution.
As a result, instead of using the BIC-type quantity, we adopt the conventional \textit{k}-fold cross-validation~\cite{bickel2008regularized}, and set $k=5$ in our experiments. For very large problem, say $p=2500$, we adopt the $3$-fold cross-validation.

For each hyperparameter, we set different search scopes, respectively.
In our experiments,
we fix $\lambda_1 = 0.4$, and consider the cases where $\lambda_3 \in \{1, 1.5, 2\}$ and $\lambda_5 \in \{0.5, 1\}$.
For the tuning of $\lambda_2$ and $\lambda_4$, we use a relatively finer grid of their values.
For details, we select $\lambda_2$ from $[0.1, 0.5]$ and $\lambda_4$ from $[0.05, 0.15]$.

Another problem in the hyperparameter tuning is to specify the index set $\mathcal{D}$, which indicates the possible hubs.
Li, Bai, and Zhou \cite{li2018learning} recommended a two-stage  tuning process, which employs the GL or HGL before applying the DHGL to identify the hubs previously, and designed two individual strategies towards various cases where the hubs are known or unknown.
For completeness, we include their proposed strategies in Algorithms~\ref{alg:4} and~\ref{alg:5}.
Furthermore,
we note that the proposed strategies also alleviate the difficulty in hyperparameter selections
because the setting of the hyperparameters in the HGL can be directly adapted to the DHGL.

\begin{algorithm}[htb] \algsetup{linenosize=\small} \small
\renewcommand{\algorithmicrequire}{\textbf{Input:}}
\renewcommand{\algorithmicensure}{\textbf{Output:}}
\caption{{\small DHGL with known hubs}}
\label{alg:4}
\begin{algorithmic}[1]
\STATE Using the HGL with cross-validation to obtain the estimated hubs, denoted by $\mathcal{H}_{\rm HGL}$.
\STATE Set $\mathcal{D} = \mathcal{K} \setminus \mathcal{H}_{\rm HGL}$, where $\mathcal{K}$ is the set of known hubs.
\STATE If $\mathcal{D} \neq \emptyset$, get the estimation $\Theta$ by solving the DHGL, and denote the estimated hubs by $\mathcal{H}_{\rm DHGL}$, where $\lambda_1, \lambda_2, \lambda_3$ are directly set by the same values as those in the HGL and $\lambda_4, \lambda_5$ are selected based on the cross-validation.
Then let $\mathcal{H}_r = \mathcal{H}_{\rm HGL} \cup \mathcal{H}_{\rm DHGL}$ as the set of the estimated hubs.
If $\mathcal{D} = \emptyset$, use the estimation $\Theta$ of the HGL as the final estimation result, and let $\mathcal{H}_r = \mathcal{H}_{\rm HGL}$.
\end{algorithmic}
\end{algorithm}

\begin{algorithm}[htb] \algsetup{linenosize=\small} \small
\renewcommand{\algorithmicrequire}{\textbf{Input:}}
\renewcommand{\algorithmicensure}{\textbf{Output:}}
\caption{{\small DHGL with unknown hubs}}
\label{alg:5}
\begin{algorithmic}[1]
\STATE Using the HGL with cross-validation to obtain the estimated hubs, denoted by $\mathcal{H}_{\rm HGL}$.
\STATE (Prior Information Construction) Tune the regularization parameter $\lambda$ of the GL from large to small until $|\mathcal{H}_{\rm GL} \setminus \mathcal{H}_{\rm HGL}| > 0$ and $|\mathcal{H}_{\rm GL} \cup \mathcal{H}_{\rm HGL}| \leq \max\{|\mathcal{H}_{\rm HGL}+a, b | \mathcal{H}_{\rm HGL}|\}$, where $a, b$ are typically set to be $2$ and $1.1$.
\STATE Set $\mathcal{D} = \mathcal{H}_{\rm GL} \setminus \mathcal{H}_{\rm HGL}$ which is nonempty.
\STATE Use the DHGL to estimate $\Theta$, where $\lambda_1, \lambda_2, \lambda_3$ remain the same values as those in the HGL, $\lambda_4$ is equally set to be $\lambda_2$, and $\lambda_5$ is selected using the cross-validation.
\end{algorithmic}
\end{algorithm}

\subsection{Experiments on synthetic data}
\label{subsec:synthetic-xperiments}

We consider two cases (the hubs are known or unknown) and apply the corresponding strategies. We apply three different algorithms to solve the core optimization problems (i.e., HGL and DHGL).

Now, we introduce three simulation settings used for the evaluation.
We first randomly generate a $p \times p$ adjacency matrix on the basis of the following three set-ups.
\begin{itemize}
\item[I -] Network with hub nodes.
We randomly select $|\mathcal{H}|$ nodes as the hubs.
For any $i < j$, we set the probability of $A_{ij} = 1$ to be $0.7$ if the corresponding node is the hub and to be $0.02$ otherwise.
That is, for any $i<j$, we have
\begin{equation*}
\mathrm{Prob}(A_{ij} = 1) =
\left\{
\begin{array}{ll}
0.7, &\quad\text{if}~i \in \mathcal{H}, \\
0.02, &\quad\text{otherwise}.
\end{array}
\right.
\end{equation*}
Finally, the adjacency matrix $A$ is completed  by setting $A_{ji} = A_{ij}$ for $i=1,\dots,p$.
\item[I\!I -] Network consisting of two connected subnetworks with hub nodes.
The adjacency matrix $A = \begin{pmatrix} A_1 & 0 \\ 0 & A_2 \end{pmatrix}$, where $A_1$ and $A_2$ are generated in accordance with the set-up I.
\item[I\!I\!I -] Scale-free network. For any given node, the probability that it has $k$ edges is proportional to $k^{-\alpha}$.
For this simulation setting, we use the graph library NetworkX \cite{hagberg2008exploring} to generate the adjacency matrix. Here we set $\alpha=2.5$.
\end{itemize}

Next, we use the adjacency matrix $A$ to create a matrix $E$.
To be specific, for any $(i,j)$ such that $A_{ij} \neq 0$, let $E_{ij}$ follow a uniform distribution with the interval $[-0.75,-0.25] \cup[0.25,0.75]$, i.e.,
in other words,
\begin{equation*}
E_{i j} \stackrel{\text {i.i.d.}}{\sim}\left\{\begin{array}{ll}
0, &\quad \text {if}~A_{ij}=0, \\
\operatorname{Unif}([-0.75,-0.25] \cup[0.25,0.75]), &\quad \text{otherwise.}
\end{array}\right.
\end{equation*}

\begin{figure}[h]
	\begin{center}
		\subfigure[Set-up I]{
			\includegraphics[width=.35\columnwidth]{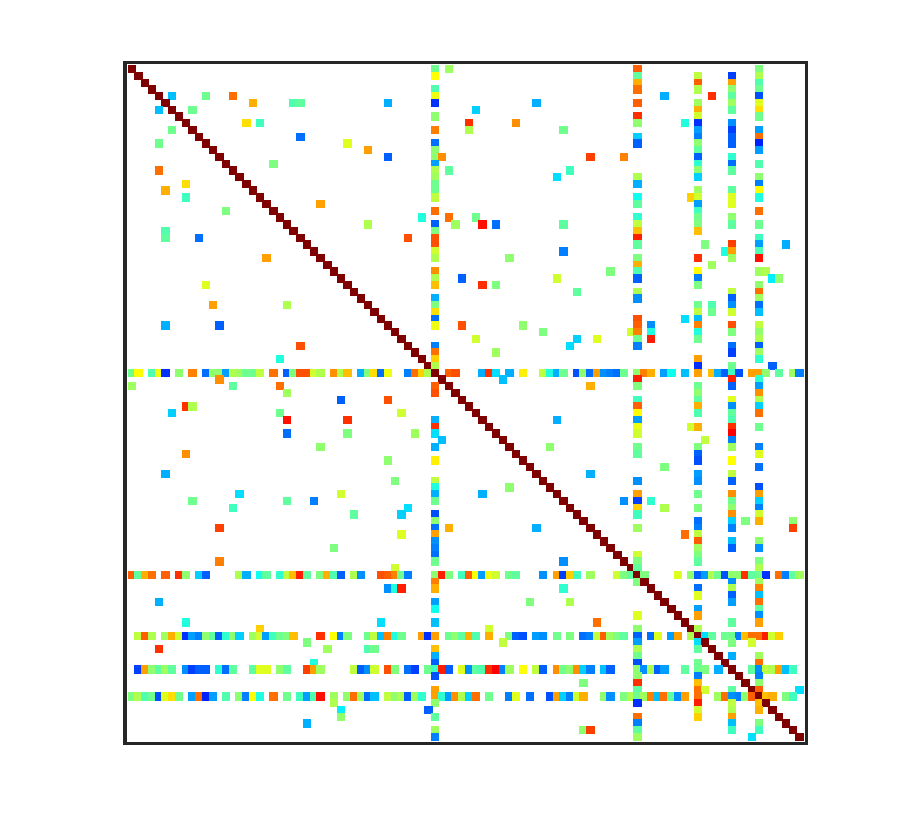}}
		\subfigure[Set-up I\!I]{
			\includegraphics[width=.35\columnwidth]{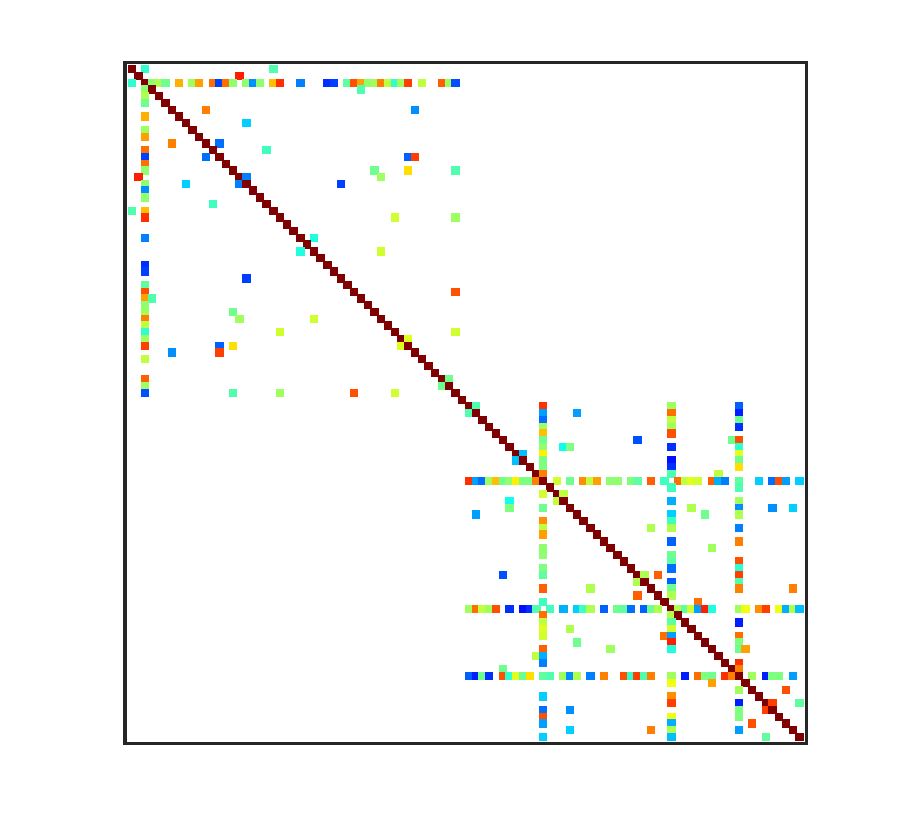}}
		\subfigure[Set-up I\!I\!I]{
			\includegraphics[width=.35\columnwidth]{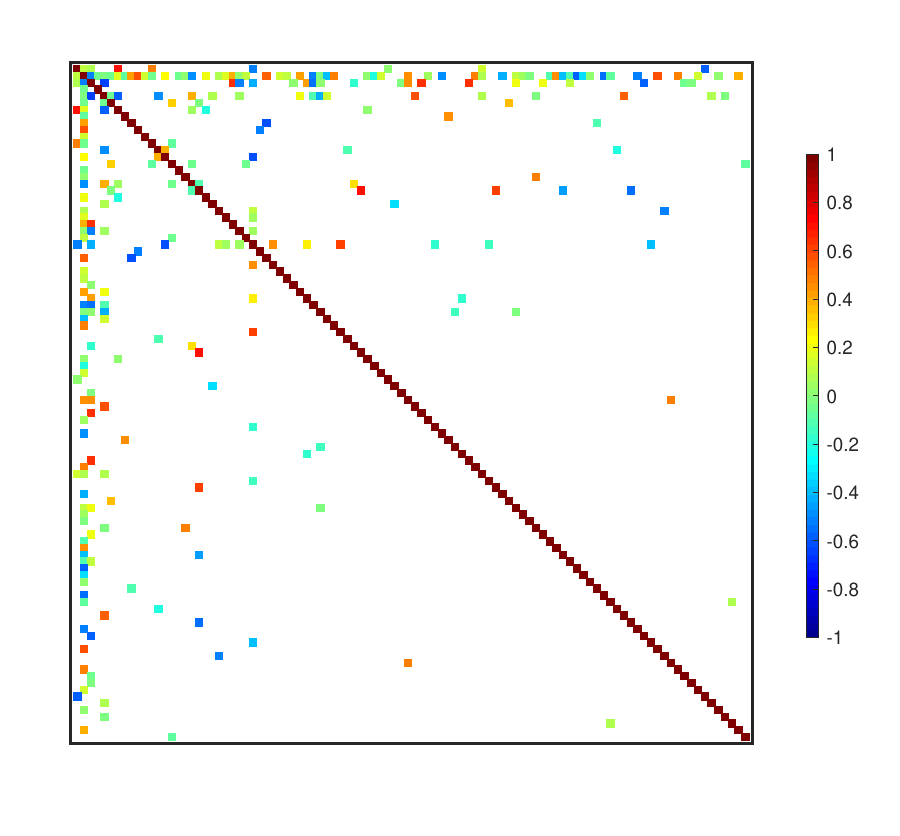}}
		\caption{Examples of the inverse covariance matrix in three set-ups.}
		\label{fig:adj_exp}
	\end{center}
\end{figure}

Then, we can obtain a symmetric matrix $\overline{E}$ by setting $\overline{E} = (E+E^T)/2$.
The inverse covariance matrix is constructed by setting $\Sigma^{-1} = \overline{E}+ (0.1 - \rho_{\min}(\overline{E}))I$, where $\rho_{\min}(\overline{E})$ indicates the smallest eigenvalue of $\overline{E}$.
Now, we can generate the data matrix $X=(\mathbf{x}_1; \dots; \mathbf{x}_n)$, where each row is sampled from the normal distribution $\mathcal{N}(0, \Sigma)$.
Finally, we conduct an additional standardization, such that the standard deviation of each feature is equal to one.
To clearly illustrate the generated inverse covariance matrix $\Sigma^{-1}$, we also give an example in Figure~\ref{fig:adj_exp} for each set-up.

For the generated data matrix $X \in \mathbb{R}^{n \times p}$, we considered the following seven cases:
\begin{eqnarray*}
(n, p)\in\Big\{(100, 300), (300, 500), (500, 800), (800, 1000),  (1000, 1500),&\\ (1500, 2000), (2000, 2500) \Big\}.&
\end{eqnarray*}
We also need to set the number of the true hub nodes
(i.e., the cardinality of the hub index set $|\mathcal{H}|$)
in set-up I and I\!I.
In our experiments, we set $|\mathcal{H}|$ to be $5$, $5$, $10$, $10$,  $30$, $30$ and $30$, respectively for each case.
Besides, the cardinality $|\mathcal{H}|$ in set-up I\!I\!I is dependent on the generated adjacency matrix rather than our setting.
In this simulation setting, a node will be identified as a hub if it has more than $r$ edges, i.e.,
$\mathcal{H} = \left\{ i\,\left|\, \sum_{j=1, j \neq i}^p A_{ij} > r,  i=1,\dots, p \right.\right\}$, and we also set $r = p/5$.

\subsubsection{Experiments with known hub nodes}

When the hub nodes are known, for the DHGL problem, the numerical performances of the pADMM, dADMM and two-phase algorithms on the synthetic data of simulation settings I, I\!I and I\!I\!I are compared. The specific numerical performances are listed in the tables below, where ``P'', ``D'', ``T'', ``T(I)'' and ``T(I\!I)'' denote the pADMM, dADMM, two-phase algorithm, Phase I algorithm and Phase I\!I algorithm, respectively, and the computational time is recorded in the format of ``hours:minutes:seconds''. For ``iter" of T(I\!I), the number in the parenthesis denotes the number of inner iterations of Phase I\!I.  As can be seen from the numerical test results in Tables 1-3, for the synthetic data of known hubs, the two-phase algorithm is more efficient than the pADMM and dADMM in all the cases.
\begin{table}[H]
	\setlength{\tabcolsep}{.35em}
	\footnotesize
	\parbox{0.98\textwidth}
	\centering
	\caption{In simulation setting I, the test results for the DHGL problem in which the hub nodes are known}
	\resizebox{0.98\textwidth}{15mm}{
		\begin{tabular}{|c|c|c|c|c|c|c|c|c|c|c|}
			\hline
			\multirow{2}*{$(n, p)$} & \multicolumn{3}{c|}{max\{$R_{P},R_{D},R_{C}$\}} & \multicolumn{4}{c|}{iter} & \multicolumn{3}{c|}{time}\\
			\cline{2-11}
			& P & D & T  & P & D & T(I) & T(I\!I) & P & D & T \\
			\hline
			$(100, 300)$ & 9.84e-07 & 9.90e-07 & 1.29e-07 & 471 & 594 & 200 & 12(19) & 0:00:11 & 00:00:18 & 00:00:10 \\
			\hline
			$(300, 500)$ & 9.34e-07 & 9.89e-07 & 1.14e-07  & 518 & 329 & 126 & 7(8) & 00:00:53 & 00:00:39 & 00:00:19\\
			\hline
			$(500, 800)$ & 9.33e-07 & 9.93e-07 & 1.91e-07  & 781 & 339 & 135 & 5(6) & 00:04:13 & 00:02:25 & 00:01:09\\
			\hline
			$(800, 1000)$ & 9.81e-07 & 9.70e-07 & 9.80e-07 & 955 & 366 & 126 & 4(5) & 00:08:42 & 00:06:02 & 00:01:35\\
			\hline
			$(1000, 1500)$ & 9.90e-07 & 9.94e-07 & 8.87e-07  & 1378 & 900 & 200 & 37(45) & 00:33:21 & 00:27:34 & 00:19:28\\
			\hline
			$(1500, 2000)$ & 1.00e-06 & 9.83e-07 & 5.01e-07 & 1609 & 646 & 200 & 12(13) & 01:19:47 & 00:39:53 & 00:18:16\\
			\hline
			$(2000, 2500)$ & 9.99e-07 & 9.96e-07 & 5.12e-07& 2857 & 1693 & 200 & 31(44) & 04:07:22 & 03:09:18 & 01:06:12\\
			\hline
		\end{tabular}
		\label{tab:exp_k1}
	}
\end{table}

\begin{table}[H]
	\setlength{\tabcolsep}{.35em}
	\footnotesize
	\parbox{0.98\textwidth}
	\centering
	\caption{In simulation setting I\!I, the test results for the DHGL problem in which the hub nodes are known}
	\resizebox{0.98\textwidth}{15mm}{
		\begin{tabular}{|c|c|c|c|c|c|c|c|c|c|c|}
			\hline
			\multirow{2}*{$(n, p)$} & \multicolumn{3}{c|}{max\{$R_{P},R_{D},R_{C}$\}}  & \multicolumn{4}{c|}{iter} & \multicolumn{3}{c|}{time}\\
			\cline{2-11}
			& P & D & T & P & D & T(I) & T(I\!I) & P & D & T \\
			\hline
			$(100, 300)$ & 9.99e-07 & 9.89e-07 & 5.62e-07  & 490 & 317 & 200& 12(16)  & 00:00:12 & 00:00:10 & 00:00:07\\
			\hline
			$(300, 500)$  & 9.96e-07 & 9.84e-07 & 8.21e-07  & 683 & 372 & 200 & 12(14)  & 00:01:15 & 00:00:44 & 00:00:31\\
			\hline
			$(500, 800)$ & 9.99e-07 & 9.89e-07 & 3.85e-07   & 781 & 484 & 200 & 13(22)  & 00:04:16 & 00:03:11 & 00:02:15\\
			\hline
			$(800, 1000)$ & 9.96e-07 & 9.84e-07 & 7.19e-07   & 1056 & 551 & 200 & 13(20) & 00:09:30 & 00:06:11 & 00:03:41 \\
			\hline
			$(1000, 1500)$ & 9.95e-07 & 9.91e-07 & 4.71e-07   & 1293 & 908 & 200 & 15(23) & 00:31:02 & 00:27:12 & 00:11:21\\
			\hline
			$(1500, 2000)$ & 9.98e-07 & 9.96e-07 & 7.95e-07   & 2012 & 1112 & 200 & 15(16) & 01:50:30 & 01:09:27 & 00:17:12\\
			\hline
			$(2000, 2500)$ & 9.98e-07 & 9.99e-07 & 5.13e-07  & 1764 & 1367 & 200 & 18(34) & 02:55:01 & 02:36:06 & 01:04:21\\
			\hline
		\end{tabular}
		\label{tab:exp_k2}
	}
\end{table}

\begin{table}[H]
	\setlength{\tabcolsep}{.35em}
	\footnotesize
	\parbox{0.98\textwidth}
	\centering
	\caption{In simulation setting I\!I\!I, the test results for the DHGL problem in which the hub nodes are known}
	\resizebox{0.98\textwidth}{15mm}{
		\begin{tabular}{|c|c|c|c|c|c|c|c|c|c|c|}
			\hline
			\multirow{2}*{$(n, p)$} & \multicolumn{3}{c|}{max\{$R_{P},R_{D},R_{C}$\}}  & \multicolumn{4}{c|}{iter} & \multicolumn{3}{c|}{time}\\
			\cline{2-11}
			& P & D & T & P & D & T(I) & T(I\!I) & P & D & T \\
			\hline
			$(100, 300)$  & 9.89e-07 & 9.77e-07 & 3.59e-07  & 449 & 338 & 200 & 13(16)  & 00:00:10 & 00:00:11 & 00:00:09\\
			\hline
			$(300, 500)$  & 9.99e-07 & 9.82e-07 & 2.67e-07   & 714 & 359 & 200 & 12(13)  & 00:01:15 & 00:00:49 & 00:00:35\\
			\hline
			$(500, 800)$  & 9.93e-07 & 9.90e-07 & 5.01e-07  & 688 & 505 & 200 & 12(12)  & 00:03:43 & 00:03:54 & 00:01:52\\
			\hline
			$(800, 1000)$  & 9.96e-07 & 9.84e-07 & 3.60e-07 & 1058 & 546 & 200 & 13(13) & 00:09:52 & 00:09:19 & 00:03:13 \\
			\hline
			$(1000, 1500)$ & 9.93e-07 & 9.97e-07 & 3.13e-07  & 778 & 606 & 200 & 14(20) & 00:20:18 & 00:21:09 & 00:09:50\\
			\hline
			$(1500, 2000)$ & 9.98e-07 & 9.96e-07 & 4.09e-07   & 1401 & 968 & 200 & 14(14) & 01:15:24 & 01:22:27 & 00:21:02\\
			\hline
			$(2000, 2500)$ & 9.95e-07 & 9.94e-07 & 5.07e-07   & 2609 & 1402 & 200 & 17(21) & 04:30:37 & 03:48:58 & 01:11:18\\
			\hline
		\end{tabular}
		\label{tab:exp_k3}
	}
\end{table}

\subsubsection{Results with unknown hub nodes}

When the hub nodes are unknown, for the DHGL problem, the numerical performances of the pADMM, dADMM and two-phase algorithms on the synthetic data of simulation Settings I, I\!I and I\!I\!I are compared. The specific numerical performances are shown in the tables below. As can be seen from the numerical test results in Tables 4-6, for the synthetic data of unknown hubs, the two-phase algorithm is the most efficient algorithm in all the cases.
\begin{table}[H]
	\setlength{\tabcolsep}{.35em}
	\footnotesize
	\parbox{0.98\textwidth}
	\centering
	\caption{In simulation setting I, the test results for the DHGL problem of the hub node are unknown}
	\resizebox{0.98\textwidth}{15mm}{
		\begin{tabular}{|c|c|c|c|c|c|c|c|c|c|c|}
			\hline
			\multirow{2}*{$(n, p)$} & \multicolumn{3}{c|}{max\{$R_{P},R_{D},R_{C}$\}}  & \multicolumn{4}{c|}{iter} & \multicolumn{3}{c|}{time}\\
			\cline{2-11}
			& P & D & T  & P & D & T(I) & T(I\!I) & P & D & T \\
			\hline
			$(100, 300)$ & 9.92e-07 & 9.98e-07 & 2.34e-07 & 737 & 483 & 200 & 12(21) & 00:00:19 & 00:00:15 & 00:00:11 \\
			\hline
			$(300, 500)$ & 9.93e-07 & 9.92e-07 & 5.44e-07  & 1038 & 636 & 200 & 12(12) & 00:01:50 & 00:01:22 & 00:00:31\\
			\hline
			$(500, 800)$ & 9.98e-07 & 9.92e-07 & 3.94e-07  & 1443 & 954 & 200 & 14(14) & 00:07:55 & 00:06:55 & 00:01:50\\
			\hline
			$(800, 1000)$ & 9.96e-07 & 8.96e-07 & 7.50e-07  & 1610 & 976 & 200 & 14(14) & 00:15:02 & 00:12:10 & 00:04:17\\
			\hline
			$(1000, 1500)$ & 1.00e-06 & 9.95e-07 & 2.52e-07  & 2238 & 1579 & 200 & 19(24) & 00:55:42 & 00:51:59 & 00:12:12\\
			\hline
			$(1500, 2000)$ & 9.98e-07 & 9.95e-07 & 4.43e-07  & 2784 & 2003 & 200 & 17(17) & 02:27:29 & 02:51:18 & 00:21:21\\
			\hline
			$(2000, 2500)$ & 9.99e-07 & 1.00e-06 & 4.96e-07  & 3954 & 2505 & 200 & 19(53) & 06:31:37 & 05:45:03 & 02:23:26\\
			\hline
		\end{tabular}
		\label{tab:exp_unk1}
	}
\end{table}

\begin{table}[H]
	\setlength{\tabcolsep}{.35em}
	\footnotesize
	\parbox{0.98\textwidth}
	\centering
	\caption{In simulation setting I\!I, the test results for the DHGL problem of the hub node are unknown}
	\resizebox{0.98\textwidth}{15mm}{
		\begin{tabular}{|c|c|c|c|c|c|c|c|c|c|c|}
			\hline
			\multirow{2}*{$(n, p)$} & \multicolumn{3}{c|}{max\{$R_{P},R_{D},R_{C}$\}}  & \multicolumn{4}{c|}{iter} & \multicolumn{3}{c|}{time}\\
			\cline{2-11}
			& P & D & T & P & D & T(I) & T(I\!I) & P & D & T \\
			\hline
			$(100, 300)$ & 9.87e-07 & 9.93e-07 & 4.54e-07  & 658 & 418 & 200& 11(11)  & 00:00:16 & 00:00:12 & 00:00:08\\
			\hline
			$(300, 500)$  & 9.95e-07 & 9.84e-07 & 4.40e-07  & 861 & 372 & 200 & 12(19)  & 00:01:29 & 00:00:45 & 00:00:36\\
			\hline
			$(500, 800)$ & 9.92e-07 & 9.89e-07 & 2.06e-07   & 929 & 484 & 200 & 13(13)  & 00:04:57 & 00:03:21 & 00:01:38\\
			\hline
			$(800, 1000)$ & 9.97e-07 & 9.84e-07 & 3.79e-07    & 1043 & 551 & 200 & 13(13) & 00:12:52 & 00:09:35 & 00:02:48 \\
			\hline
			$(1000, 1500)$ & 9.94e-07 & 9.91e-07 & 2.51e-07  & 1531 & 908 & 200 & 15(15) & 00:39:33 & 00:32:09 & 00:07:34\\
			\hline
			$(1500, 2000)$ & 9.96e-07 & 9.96e-07 & 4.25e-07  & 1860 & 1112 & 200 & 15(15) & 01:43:56 & 01:15:23 & 00:15:25\\
			\hline
			$(2000, 2500)$ & 9.98e-07 & 9.99e-07 & 2.77e-07 & 2446 & 1367 & 200 & 18(18) & 04:01:29 & 02:59:30 & 00:34:40\\
			\hline
		\end{tabular}
		\label{tab:exp_unk2}
	}
\end{table}

\begin{table}[H]
	\setlength{\tabcolsep}{.35em}
	\footnotesize
	\parbox{0.98\textwidth}
	\centering
	\caption{In simulation setting I\!I\!I, the test results for the DHGL problem of the hub node are unknown}
	\resizebox{0.98\textwidth}{15mm}{
		\begin{tabular}{|c|c|c|c|c|c|c|c|c|c|c|}
			\hline
			\multirow{2}*{$(n, p)$} & \multicolumn{3}{c|}{max\{$R_{P},R_{D},R_{C}$\}} & \multicolumn{4}{c|}{iter} & \multicolumn{3}{c|}{time}\\
			\cline{2-11}
			& P & D & T  & P & D & T(I) & T(I\!I) & P & D & T \\
			\hline
			$(100, 300)$ & 9.90e-07 & 9.99e-07 & 5.18e-07  & 438 & 259 & 200 & 10(10)  & 00:00:10 & 00:00:09 & 00:00:08\\
			\hline
			$(300, 500)$  & 9.92e-07 & 9.95e-07 & 4.67e-07   & 591 & 380 & 200 & 10(16)  & 00:01:01 & 00:00:53 & 00:00:42\\
			\hline
			$(500, 800)$ & 9.92e-07 & 9.83e-07 & 2.34e-07   & 1051 & 485 & 200 & 12(12)  & 00:05:39 & 00:04:18 & 00:01:53\\
			\hline
			$(800, 1000)$ & 9.96e-07 & 9.94e-07 & 3.42e-07   & 959 & 557 & 200 & 12(12) & 00:08:53 & 00:06:57 & 00:03:29 \\
			\hline
			$(1000, 1500)$ & 9.97e-07 & 9.96e-07 & 3.18e-07   & 1341 & 937 & 200 & 14(14) & 00:35:25 & 00:39:04 & 00:09:12\\
			\hline
			$(1500, 2000)$ & 9.98e-07 & 9.92e-07 & 4.34e-07   & 1470 & 853 & 200 & 15(15) & 01:22:28 & 01:14:32 & 00:22:33\\
			\hline
			$(2000, 2500)$ & 9.98e-07 & 9.89e-07 & 2.80e-07 & 2596 & 793 & 200 & 16(16) & 04:06:21 & 01:30:01 & 00:29:11\\
			\hline
		\end{tabular}
		\label{tab:exp_unk3}
	}
\end{table}

In order to explain the graph recovery of the three algorithms more intuitively, we select an example with dimension $p=300$ for the simulation setting I in Section \ref{subsec:synthetic-xperiments}. As shown in Figure \ref{fig:res_exp}, the ground truth matrix contains five hubs, and the estimations obtained via pADMM, dADMM, and the two-phase algorithm all correctly identify five hubs, as they are solving the same DHGL model.

\begin{figure}[h]
	\begin{center}
		\subfigure[ ]{
			\includegraphics[width=.35\columnwidth]{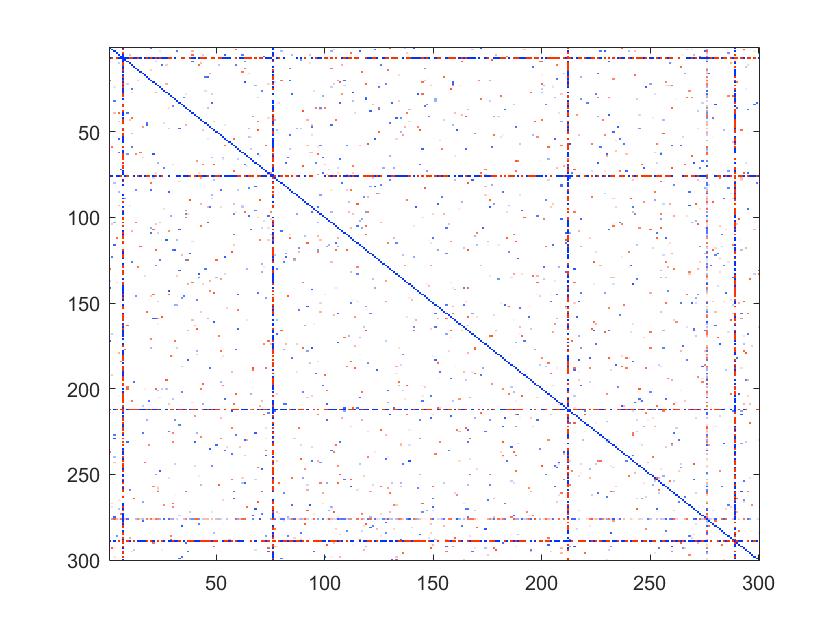}}
		\subfigure[ ]{
			\includegraphics[width=.35\columnwidth]{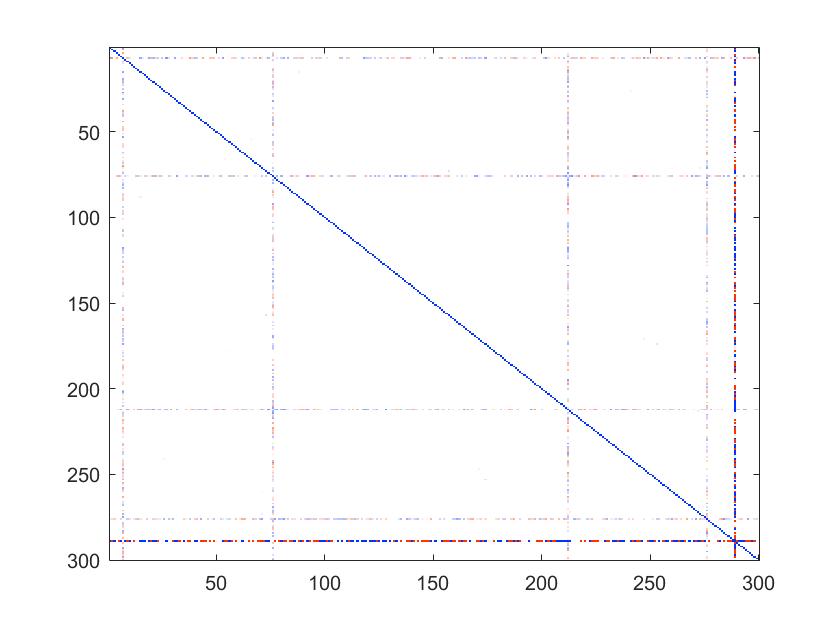}}
		\subfigure[ ]{
			\includegraphics[width=.35\columnwidth]{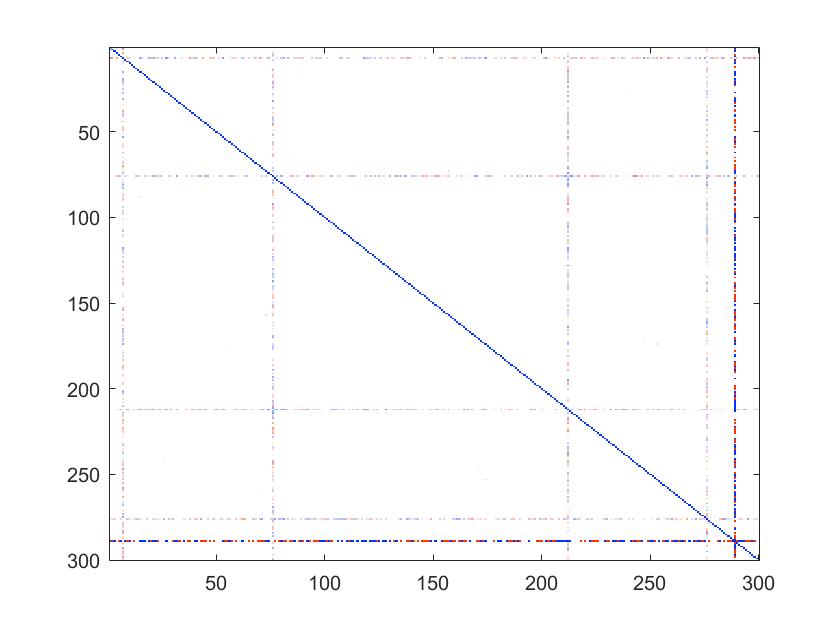}}
        \subfigure[ ]{
			\includegraphics[width=.35\columnwidth]{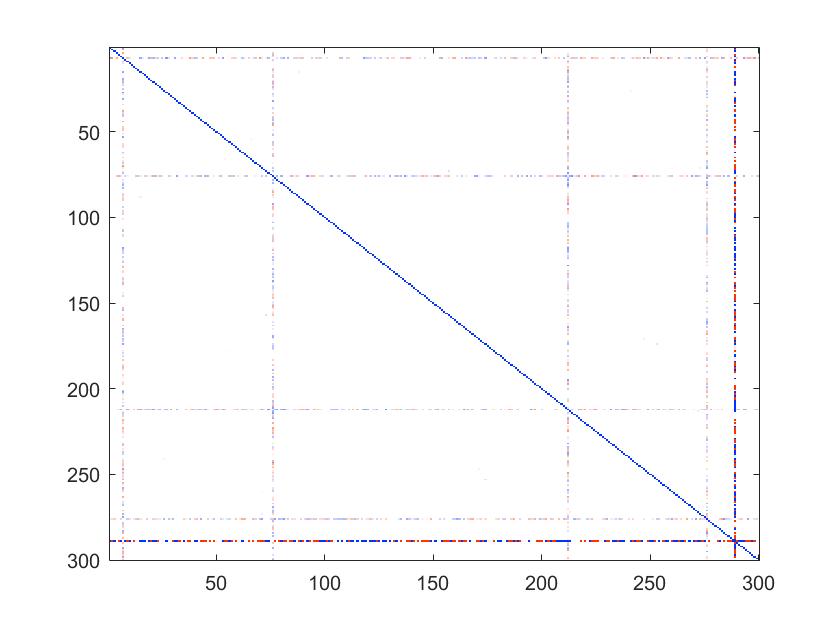}}
\caption{The adjacency matrix of graph estimation by using three algorithms.
The four subfigures are: (a) the true adjacency matrix (i.e., ground truth), (b) the estimated result by the pADMM, (c) the estimated result by the dADMM, and (d) the estimated result by the two-phase algorithm, respectively.
}
		\label{fig:res_exp}
	\end{center}
\end{figure}

\begin{figure}[H]
	\centering
	\vspace{-0.55in}
	\begin{minipage}{1\linewidth}
		\includegraphics[width=0.45\linewidth]{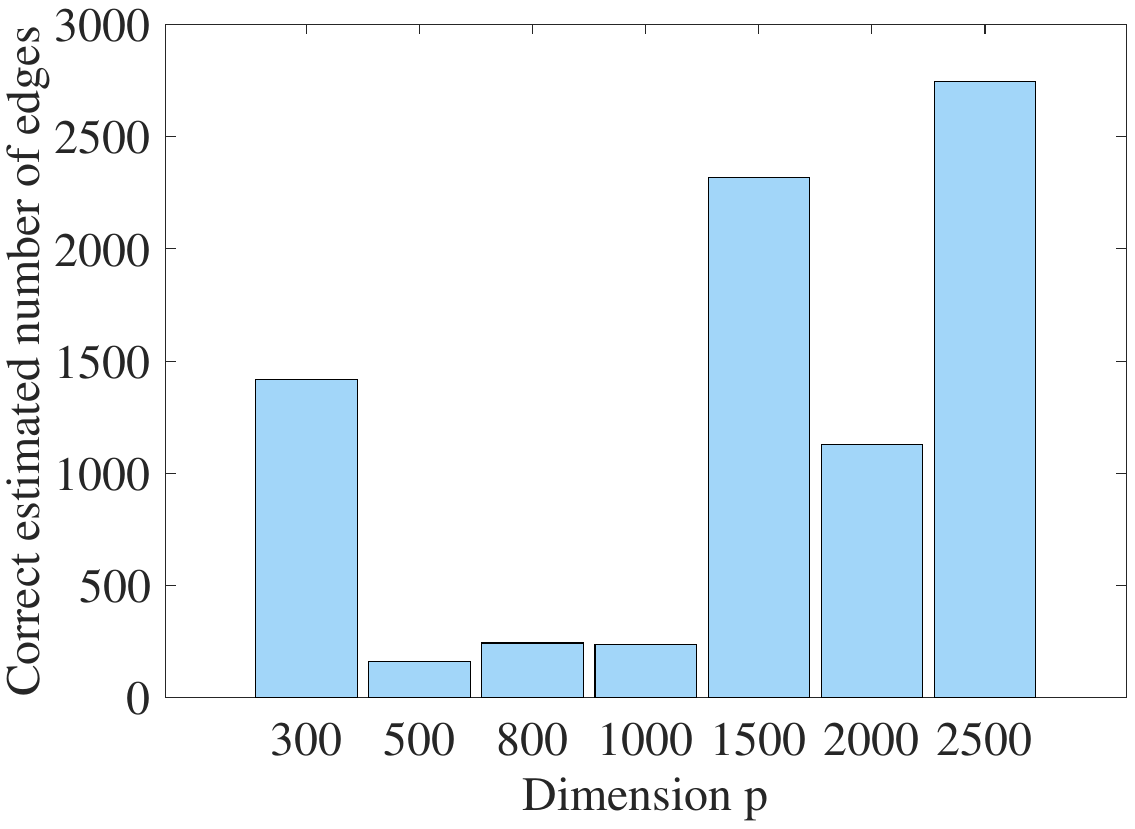}
		\includegraphics[width=0.45\linewidth]{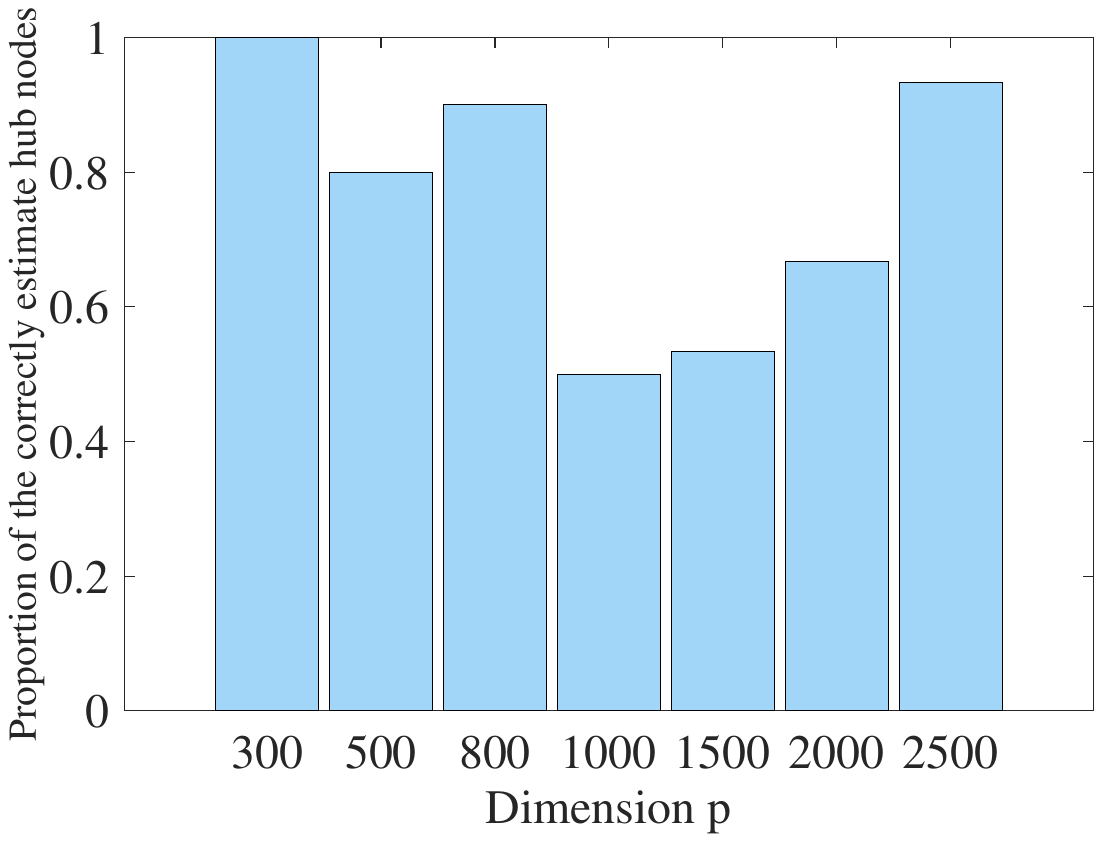}
	\end{minipage}
	\vskip -3.5cm
	\begin{minipage}{1\linewidth}
		\includegraphics[width=0.45\linewidth]{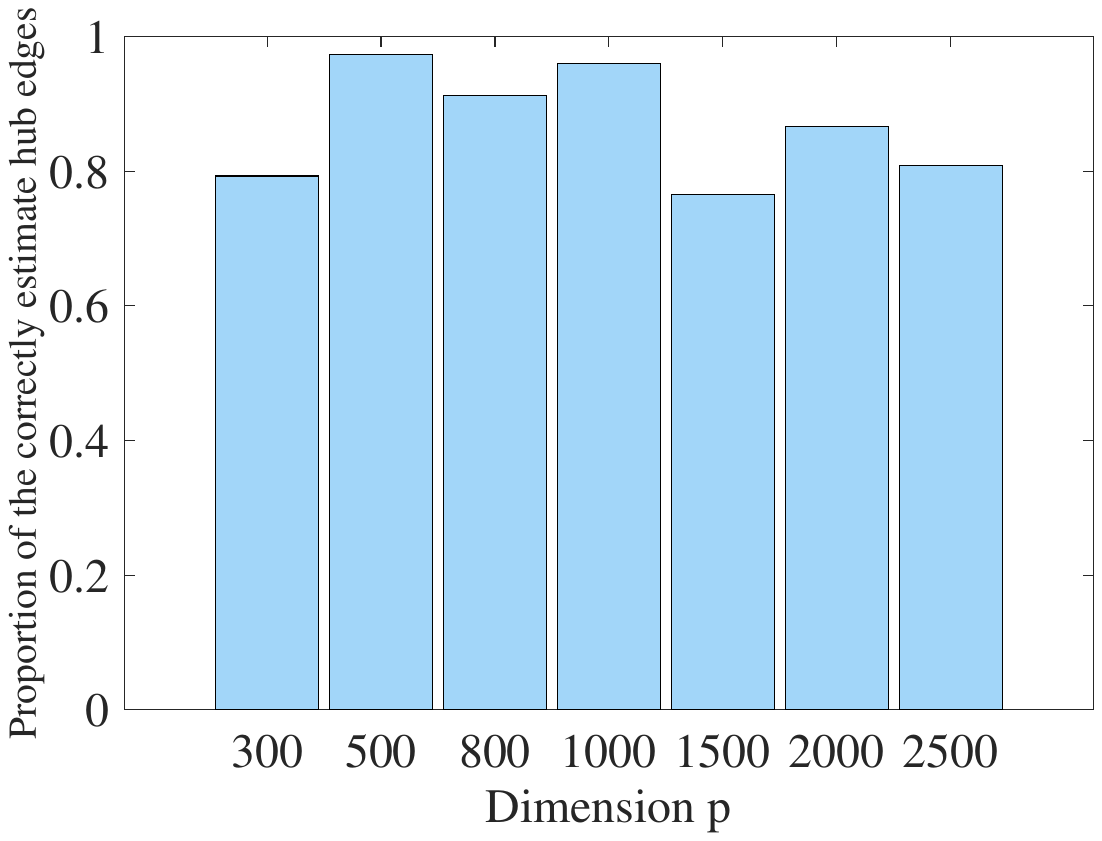}
		\includegraphics[width=0.45\linewidth]{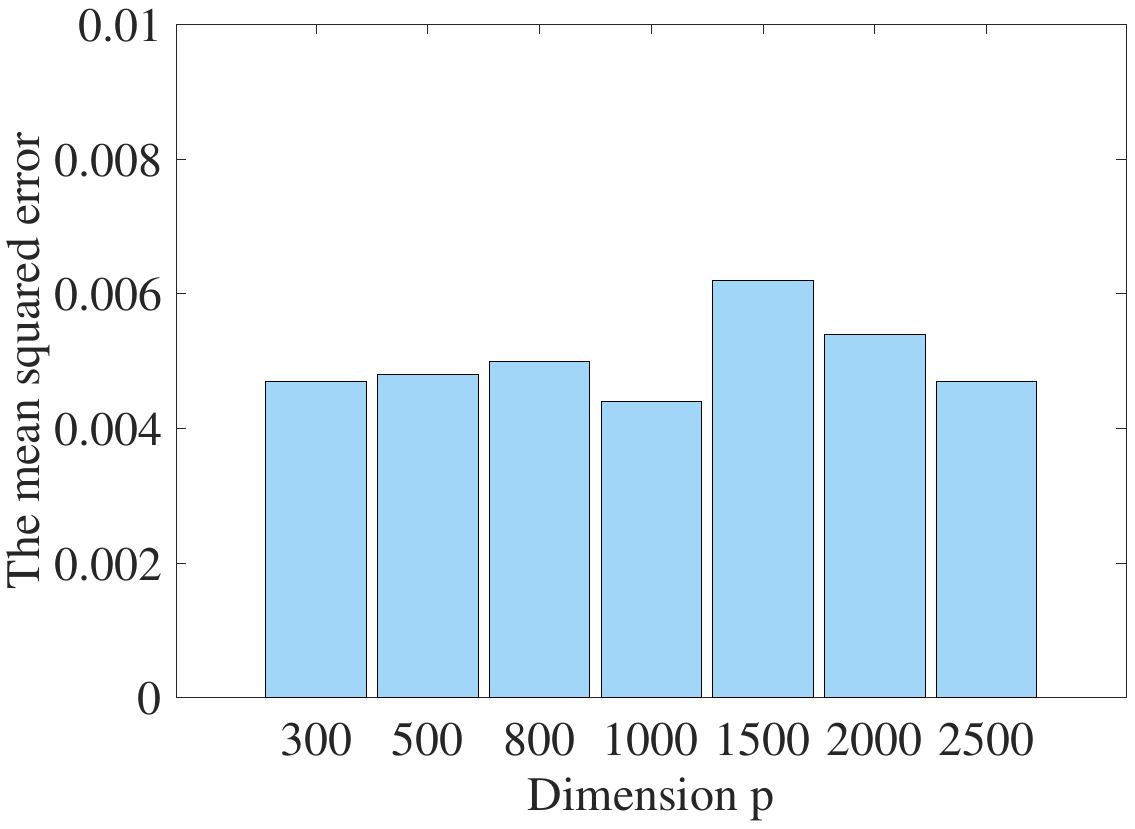}
	\end{minipage}
	\vspace{-0.58in}
	\caption{Results of synthetic data efficacy measures}\label{fig: eff}
		\vspace{-0.2in}
\end{figure}

Meanwhile, in order to illustrate the efficacy of the two-phase algorithm, we draw the relevant data of four criteria (the same as those in \cite{tan2014learning}) to evaluate the efficacy of the algorithm in seven different dimensions for the simulation setting I when the hub nodes are known. As can be seen from Figure \ref{fig: eff}, the two-phase algorithm performs well in the correct estimation of the number of edges, the proportion of correctly estimated hub nodes, the proportion of correctly estimated hub edges, and the mean squared error, which also fully demonstrate the efficacy of the two-phase algorithm.

\subsection{Experiments on real-world data}

In this subsection, we work on two sets of read data to demonstrate the efficiency of the proposed algorithm.

\begin{figure}[H]
	\includegraphics[width=0.95\columnwidth]{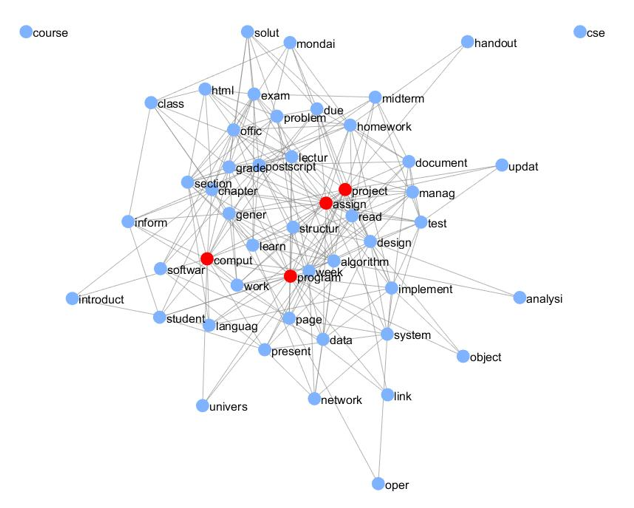}
	
	\caption{The resulting network of student webpages data. The nodes represent 50 words. Detected hub nodes are marked in red.}
	
	\label{fig:student-100}
\end{figure}
\subsubsection{University webpage data}

In this part, we assess the numerical performance of different algorithms using the university webpages dataset, available at http://ana.cachopo.org/datasets-for-single-label-text-categorization. The original dataset consists of webpages collected from the computer science departments of several universities in 1997, including Cornell, Texas, Washington, and Wisconsin. It records the occurrences of various terms (words) found on these webpages. We focus solely on the 544 student webpages and choose 100 terms with the highest entropy for our analysis. Subsequently, we represent these 100 terms as nodes in a Gaussian graphical model.

The aim of the analysis is to explore the relationships between the terms found on the student webpages. Specifically, we seek to identify terms that serve as hubs.
In Table 7, we present the computational results of three algorithms. As observed, the two-phase algorithm continues to outperform both pADMM and dADMM, even though the problem is quite small and all three algorithms are capable of achieving the desired accuracy. In Figure \ref{fig:student-100}, we plot the resulting network. For clarity, we only select 50 nodes to plot the network, including the 40 nodes with the most links and the 10 nodes with the least links. The four most connected nodes, identified as hubs, are highlighted in red.

\begin{table}[H]
	\setlength{\tabcolsep}{.35em}
	\footnotesize
	\parbox{0.98\textwidth}
	\centering
	\caption{For the university webpage data, the test results of the DHGL problem}
	\resizebox{\textwidth}{!}{
		\begin{tabular}{|c|c|c|c|c|c|c|c|c|c|c|}
			\hline
			\multirow{2}*{$(n, p)$} & \multicolumn{3}{c|}{max\{$R_{P},R_{D},R_{C}$\}} & \multicolumn{4}{c|}{iter} & \multicolumn{3}{c|}{time}\\
			\cline{2-11}
			& P & D & T  & P & D & T(I) & T(I\!I) & P & D & T \\
			\hline
			student$(544,100)$ & 9.86e-07 & 5.62e-07 & 3.15e-07  & 498 & 544 & 200 & 11(22)  & 00:00:02 & 00:00:02 & 00:00:01\\
			\hline
		\end{tabular}
		\label{tab:webpage}
	}
\end{table}

\subsubsection{Portfolio data}
\label{subsubsec:portfolios data}

In this part, we compare three algorithms on the portfolio data formed on size and operating profitability, which is downloaded from https://mba.tuck.dartm outh.edu/pages/faculty/ken.french/data\_library.html.
The portfolios, constructed annually at the end of June, are formed by intersecting 10 size-based portfolios (market equity, ME) with 10 operating profitability (OP) portfolios. The size breakpoints for year $t$ correspond to the New York Stock Exchange (NYSE) market equity deciles as of June in year $t$. The OP for June of year $t$ is computed as annual revenues minus the cost of goods sold, interest expense, and selling, general, and administrative expenses, divided by book equity from the last fiscal year ending in $t-1$. The OP breakpoints are also determined by NYSE deciles. This dataset consists of 100 different portfolios, from 6076 sampling between 1963 and 2024. The 100 different portfolios represent 100 nodes in a Gaussian graphical model.

\begin{table}[H]
	\setlength{\tabcolsep}{.35em}
	\footnotesize
	\parbox{0.98\textwidth}
	\centering
	\caption{For the portfolio data, the test results of the DHGL problem}
	\resizebox{\textwidth}{!}{
		\begin{tabular}{|c|c|c|c|c|c|c|c|c|c|c|}
			\hline
			\multirow{2}*{$(n, p)$} & \multicolumn{3}{c|}{max\{$R_{P},R_{D},R_{C}$\}} & \multicolumn{4}{c|}{iter} & \multicolumn{3}{c|}{time}\\
			\cline{2-11}
			& P & D & T  & P & D & T(I) & T(I\!I) & P & D & T \\
			\hline
			portfolios$(6076,100)$ & 3.69e-04 & 1.78e-04 & 9.99e-07  & 10000 & 10000 & 200 & 12(25)  & 00:00:42 & 00:00:38 & 00:00:18\\
			\hline
		\end{tabular}
		\label{tab:portfolios}
	}
\end{table}

\begin{figure}[h]
	\includegraphics[width=0.95\columnwidth]{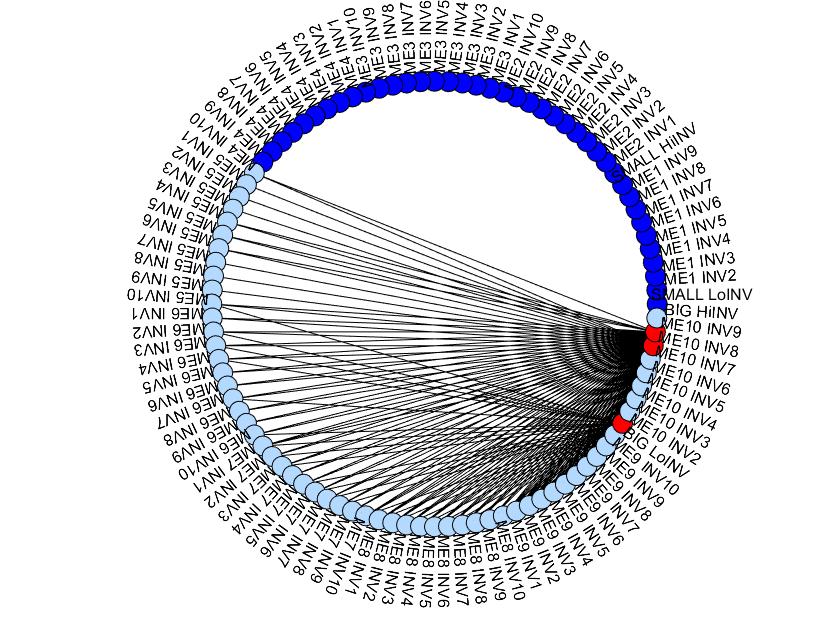}
	
	\caption{The resulting network of the portfolio data. The nodes represent 100 investment stratagies. Detected hub nodes are marked in red.}
	
	\label{fig:runportfolios-100}
\end{figure}

The purpose of this analysis is to study the correlations among the investment portfolios.
In Table 8, we present the computational results of three algorithms. As observed, the two-phase algorithm outperforms both pADMM and dADMM. In Figure \ref{fig:runportfolios-100}, we present the resulting network, where all portfolios are arranged along a large circle. Nodes with 30 or more edges are identified as hubs and highlighted in red, while isolated nodes are shown in blue. The remaining nodes are depicted in light blue.
}

\section{Conclusion}
\label{sec:Conclusion}

In this paper, we have developed a two-phase algorithm to solve the hub graphical lasso model with the structured sparsity. Specifically, we design the dADMM in Phase I to generate a good initial point to warm start Phase I\!I of the ALM. The SSN method is applied to solve the inner subproblems of the ALM. We take full advantage of the sparsity structure of the generalized Jacobian and make the performances of the SSN and ALM efficiently. Numerical experiments on
both synthetic data and real data have demonstrated the efficacy and efficiency of the proposed algorithm.

\section*{Acknowledgments}
We would like to thank the Editor-in-Chief Professor Chi-Wang Shu, the anonymous Associate Editor and referee for their helpful suggestions which greatly improves the quality of the manuscript.

\section*{Funding}
Chengjing Wang's work was supported in part by the National Natural Science Foundation of China
(No. U21A20169), Zhejiang Provincial Natural Science Foundation of China (Grant No. LTGY23H240002).
Meixia Lin's work was supported by the Ministry of Education, Singapore, under its Academic Research Fund Tier 2 grant call (MOE-T2EP20123-0013).

\section*{Data Availability}
The datasets analysed during the current study are available at the following link:\\
https://archive.ics.uci.edu/\\
https://CRAN.R-project.org/package=spectralGraphTopology\\
https://www.genomeweb.com/archive/iconix-links-drugmatrix-database-mdl-information-systems-tools\\
http://ana.cachopo.org/datasets-for-single-label-text-categorization\\
https://mba.tuck.dartmouth.edu/pages/faculty/ken.french/data\_library.html.
\section*{Declarations}
\textbf{Conflict of interest} The authors have not disclosed any competing interests.

\begin{appendices}
\section{Proof of Proposition~\ref{The composite of proximal operator}}
\label{proof of the composite of proximal operator}
\begin{proof}
	Firstly, we calculate the proximal operators of $\varphi$ and $\psi$.
	Since for any $V\in \mathbb{M}^p$,
	\begin{equation*}
		\varphi(V) = \sum_{j=1}^p w_{2,j} \|\mathcal{P}_j V\|,
	\end{equation*}
	by denoting $\widehat{\mathcal{P}}_j: \mathbb{R}^p\rightarrow \mathbb{R}^p$ as $\widehat{\mathcal{P}}_j y=[y_{1};\cdots;y_{j-1};0;y_{j+1};\cdots;y_{p}]$ for any $y\in \mathbb{R}^p$, we have that
	\begin{align*}
		\varphi^{*}(Y) &= \sup_{V\in \mathbb{M}^p}\Big\{\langle Y,V\rangle-\sum_{j=1}^p w_{2,j} \|\mathcal{P}_j V\|\Big\}\\
		&= \sum_{j=1}^{p}\sup_{V_j\in \mathbb{R}^p}\Big\{\langle Y_{j},V_{j}\rangle-w_{2,j} \|\widehat{\mathcal{P}}_j V_j\|\Big\}\\
		&= \sum_{j=1}^{p}\sup_{V_j\in \mathbb{R}^p}\Big\{Y_{jj}V_{jj}+\langle \mathcal{P}_jY,\widehat{\mathcal{P}}_j V_j\rangle-w_{2,j} \|\widehat{\mathcal{P}}_j V_j\|\Big\}.
	\end{align*}
	Note that
	\begin{align*}
		&Y_{jj}V_{jj}+\left\langle \mathcal{P}_jY,\widehat{\mathcal{P}}_j V_j\right\rangle-w_{2,j} \|\widehat{\mathcal{P}}_j V_j\|\leq Y_{jj}V_{jj}+\|\widehat{\mathcal{P}}_j V_j\|(\|\mathcal{P}_j Y\|-w_{2,j}).
	\end{align*}
	Therefore,
	\begin{eqnarray*}
		\varphi^{*}(Y) &=& \delta_{C}(Y),
	\end{eqnarray*}
	where $C:=\Big\{Y\in \mathbb{M}^{p}\,\Big|\,\|\mathcal{P}_{j}Y\| \leq w_{2,j}\textrm{ and }Y_{jj}=0, j=1,2,\cdots,p\Big\}$.
	Then by the Moreau identity, we have
	\begin{eqnarray*}
		\operatorname{prox}_{\varphi}(Y)=Y-\operatorname{prox}_{\varphi^{*}}(Y)=Y-\operatorname{\Pi}_{C}(Y)= Y - \mathcal{P}^* \operatorname{\Pi}_{\mathcal{B}^{w_2}_2}(\mathcal{P}Y),
	\end{eqnarray*}
	where
	\begin{align} \label{eq:projection-B2}
		&[\operatorname{\Pi}_{\mathcal{B}^{w_2}_{2}}(\mathcal{P} Y)]_{(j-1)p+1:jp} = \operatorname{\Pi}_{\mathcal{B}_2^{w_{2,j}}}(\mathcal{P}_j Y)=
		\left\{
		\begin{array}{cl}
			w_{2,j} \frac{\mathcal{P}_j Y}{\|\mathcal{P}_j Y\|}, &\quad \text{if}~ \|\mathcal{P}_j Y\| >  w_{2,j}, \\
			\mathcal{P}_j Y, &\quad \text{otherwise.}
		\end{array}
		\right.
	\end{align}
	
	Similarly, we can calculate that
	\begin{eqnarray*}
		\operatorname{prox}_{\psi}(Y)=Y - \mathcal{P}^*\operatorname{\Pi}_{\mathcal{B}^{w_1}_\infty}(\mathcal{P}Y),
	\end{eqnarray*}
	where
	\begin{align*}
		[\operatorname{\Pi}_{\mathcal{B}^{w_1}_{\infty}}(\mathcal{P} Y)]_{(j-1)p+1:jp} &= \operatorname{\Pi}_{\mathcal{B}_{\infty}^{w_{1,j}}}(\mathcal{P}_j Y)=  \operatorname{sign}(\mathcal{P}_j Y) \circ \min(|\mathcal{P}_j Y|, w_{1,j}).
	\end{align*}
	
	In the following, we prove the result about the proximal mapping of $R(\cdot)$. We can see from \eqref{eq:R-fun} that the function $R(V)$ has a separable structure, thus we only need to prove that
	\begin{align*}
		&\operatorname{prox}_{w_{1,j} \|\mathcal{P}_j \cdot\|_1 + w_{2,j} \|\mathcal{P}_j \cdot\|}(Y)= \operatorname{prox}_{w_{2,j} \|\mathcal{P}_j \cdot\|}\circ\operatorname{prox}_{w_{1,j} \|\mathcal{P}_j \cdot\|_{1}}(Y), \ j=1,2,\cdots,p.
	\end{align*}
	
	By \cite[Theorem 1]{yu2013decomposing}, it suffices to show that for each $j=1,\cdots,p$,
	\begin{eqnarray*}
		\partial(w_{1,j}\|\mathcal{P}_{j}Y\|_{1})\subseteq\partial(w_{1,j}\|\mathcal{P}_{j}Z\|_{1}),
	\end{eqnarray*}
	where $Z={\rm prox}_{\varphi}(Y)=Y-\mathcal{P}^{*}\operatorname{\Pi}_{\mathcal{B}_{2}^{w_2}}(\mathcal{P}Y)$. We prove this relation in two cases:
	
	(1) If $\|\mathcal{P}_j Y\| \leq w_{2,j}$, from \eqref{eq:projection-B2}, we have $\mathcal{P}_{j}Z=0$. Then according to \cite[Theorem 23.9]{Rockafellar1970Convex}, we obtain
	\begin{eqnarray*}
		\partial(w_{1,j}\|\mathcal{P}_{j}Z\|_{1})=\Big\{\mathcal{P}_{j}^{*}v\,\Big|\,v\in[-w_{1,j},w_{1,j}]^{p} \Big\},
	\end{eqnarray*}
	which contains $\partial(w_{1,j}\|\mathcal{P}_{j}Y\|_{1})$.
	
	(2) If $\|\mathcal{P}_j Y\| >  w_{2,j}$, from \eqref{eq:projection-B2}, we have $\mathcal{P}_{j}Z=\left(1-\frac{w_{2,j}}{\|\mathcal{P}_{j}Y\|}\right)(\mathcal{P}_{j}Y)$, which implies $\textrm{sign}(\mathcal{P}_j Y)=\textrm{sign}(\mathcal{P}_{j}Z)$. Therefore, $\partial(w_{1,j}\|\mathcal{P}_{j}Y\|_{1})\subseteq\partial(w_{1,j}\|\mathcal{P}_{j}Z\|_{1})$.
	
	Now we have finished all the proof.
\end{proof}
\end{appendices}

\bibliographystyle{spmpsci}
\bibliography{bib}


%
%



\end{document}